\documentclass[a4paper]{amsart}

\usepackage{graphicx}
\usepackage{amsmath}
\usepackage{amssymb}
\input xy
\xyoption{all}

\renewcommand{\mod}{\operatorname{mod}\nolimits}
\newcommand{\Fac}{\operatorname{Fac}\nolimits}
\newcommand{\add}{\operatorname{add}\nolimits}
\newcommand{\rad}{\operatorname{rad}\nolimits}
\newcommand{\Hom}{\operatorname{Hom}\nolimits}

\newcommand{\op}{\operatorname{op}\nolimits}
\newcommand{\dpth}{\operatorname{depth}\nolimits}
\newcommand{\CM}{\operatorname{CM}\nolimits}
\newcommand{\id}{\operatorname{id}\nolimits}
\newcommand{\pd}{\operatorname{pd}\nolimits}
\newcommand{\Ker}{\operatorname{Ker}\nolimits}
\newcommand{\gl}{\operatorname{gl.dim}\nolimits}
\newcommand{\Ext}{\operatorname{Ext}\nolimits}
\newcommand{\End}{\operatorname{End}\nolimits}
\newcommand{\Spec}{\operatorname{Spec}\nolimits}
\newcommand{\DTr}{\operatorname{DTr}\nolimits}
\newcommand{\ord}{\operatorname{ord }\nolimits}
\newcommand{\lar}{\longrightarrow}
\newcommand{\ul}{\underline}
\newcommand{\ol}{\overline}
\newcommand{\xto}{\xrightarrow}
\newcommand{\ga}{\Gamma}
\newcommand{\la}{\Lambda}

\newtheorem{theorem}{Theorem}[section]

\newtheorem{corollary}[theorem]{Corollary}
\newtheorem{definition}[theorem]{Definition}
\newtheorem{lemma}[theorem]{Lemma}
\newtheorem{proposition}[theorem]{Proposition}

\newcommand{\cc}{{\mathcal C}}

\begin{document}

\title{Cluster tilting for one-dimensional hypersurface singularities}

\author{Igor Burban}
\address{Johannes-Gutenberg Universit\"at Mainz,
Fachbereich Physik, Mathematik und Informatik,
Institut f\"ur Mathematik, 55099 Mainz, Germany}
\email{burban@mathematik.uni-mainz.de}

\author{Osamu Iyama}
\address{Graduate School of Mathematics, Nagoya University,
Chikusa-ku, Nagoya, 464-8602, Japan}
\email{iyama@math.nagoya-u.ac.jp}

\author{Bernhard Keller}
\address{UFR de Math\'ematiques, UMR 7586 du CNRS, Case 7012,
Universit\'e Paris 7, 2 place Jussieu, 75251 Paris Cedex 05, France}
\email{keller@math.jussieu.fr}

\author{Idun Reiten}
\address{Institutt for matematiske fag, Norges
Teknisk-naturvitenskapelige universitet, N-7491, Trondheim, Norway}
\email{idunr@math.ntnu.no}

\thanks{The first author was supported by the DFG project Bu 1866/1-1, the second and last author by a Storforsk grant 167130 from the Norwegian Research Council}
\maketitle


\begin{abstract}
In this article we study Cohen-Macaulay modules over  
one-dimensional hypersurface singularities
and the relationship with the representation theory of associative 
algebras using  methods of cluster tilting theory.
We give a criterion for existence of cluster tilting objects
and their complete description by homological methods, using
higher almost split sequences and results from birational geometry.
We obtain a large class of 2-CY tilted algebras which are
finite dimensional symmetric and satisfy $\tau^2=\id$.
In particular, we compute 2-CY tilted algebras for
simple and minimally elliptic curve singularities.
\end{abstract}

\section*{Introduction}
Motivated by the Fomin-Zelevinsky theory of cluster algebras \cite{FZ1,FZ2,FZ3}, a tilting
theory in cluster categories was initiated in \cite{BMRRT}. For a
finite dimensional hereditary algebra $H$ over a field $k$, the
associated cluster category $\mathcal{C}_{H}$ is the orbit category
$\mathcal{D}^b(H)/{F}$, where $\mathcal{D}^b(H)$ is the bounded
derived category of finite dimensional $H$-modules and the functor
$F\colon \mathcal{D}^b(H)\to \mathcal{D}^b(H)$ is $\tau^{-1}[1]=
S^{-1}[2]$. Here $\tau$ denotes the translation associated with
almost split sequences/triangles and $S$ the Serre functor \cite{BK} on
$\mathcal{D}^b(H)$. (See \cite{CCS} for an independent
definition of a category equivalent to the cluster category when $H$ is of Dynkin type $A_n$).

An object $T$ in a cluster category $\mathcal{C}_{H}$ was defined to be a
(cluster) tilting object if $\Ext_{\mathcal{C}_{H}}^1(T,T)=0$, and if
$\Ext_{\mathcal{C}_{H}}^1(X,X\oplus T)=0$, then $X$ is in $\add T$. The 
corresponding endomorphism algebras, called cluster tilted algebras, were
investigated in \cite{BMR1} and subsequent papers. A useful additional property of a
cluster tilting object was that even the weaker condition
$\Ext_{\mathcal{C}_{H}}^1(X,T)=0$ implies that $X$ is in $\add T$, called Ext-configuration in \cite{BMRRT}. Such a 
property also appears naturally in the work of the second author on a higher
theory of almost split sequences in module categories
\cite{I1,I2} and the notion corresponding to the above definition was called maximal 1-orthogonal. For the category
$\mod(\Lambda)$ of finite dimensional modules over a preprojective
algebra of Dynkin type $\Lambda$ over an algebraically closed
field $k$, the concept corresponding to the above definition of
cluster tilting object in a cluster category was called maximal
rigid \cite{GLSc}. Also in this setting it was shown that being maximal
1-orthogonal was a consequence of being maximal rigid. The same
result holds for the stable category $\underline{\mod}(\Lambda)$.

The categories $\mathcal{C}_{H}$ and $\underline{\mod}(\Lambda)$ are
both triangulated categories \cite{Ke,H}, with finite
dimensional homomorphism spaces, and they have Calabi-Yau dimension 2
(2-CY for short) (see
\cite{BMRRT,Ke}\cite[3.1,1.2]{AR}\cite{C}\cite[8.5]{Ke}). The last fact
means that there is a Serre functor 
$S=\Sigma^2$, where $\Sigma$ is the shift functor in the triangulated
category. 

For an arbitrary 2-CY triangulated category $\mathcal{C}$ with finite
dimensional homomorphism spaces over a field $k$, a cluster tilting
object $T$ in $\mathcal{C}$ was defined to be an object satisfying the
stronger property discussed above, corresponding to the property of being maximal 1-orthogonal/Ext-configuration \cite{KR}. The corresponding class of algebras, containing the cluster tilted ones, have been called 2-CY tilted. With this
concept many results  have been generalised from cluster categories, and
from the stable categories $\underline{\mod}(\Lambda)$, to this more
general setting in \cite{KR}, which moreover contains several results which
are new also in the first two cases.

One of the important applications of classical tilting theory has been
the construction of derived equivalences: Given a tilting bundle $T$
on a smooth projective variety $X$, the total right derived functor of
$\Hom(T,\text{ })$ is an equivalence from the bounded derived category
of coherent sheaves on $X$ to the bounded derived category of finite
dimensional modules over the endomorphism algebra of $T$. Analogously,
cluster tilting theory allows one to establish equivalences between
very large factor categories appearing in the local situation of
Cohen-Macaulay modules and categories of modules over
finite dimensional algebras. Namely, if $\ul{\CM}(R)$ is the stable
category of maximal Cohen-Macaulay modules over an odd-dimensional isolated
hypersurface singularity, then $\ul{\CM}(R)$ is 2-CY. If it contains a
cluster tilting object $T$, then the functor $\Hom(T,\text{ })$
induces an equivalence between the quotient of $\ul{\CM}(R)$ by the
ideal of morphisms factoring through $\tau T$ and the category of
finite dimensional modules over the endomorphism algebra
$B=\End(T)$. It is then not hard to see that $B$ is symmetric and the
indecomposable nonprojective $B$-modules are $\tau$-periodic of
$\tau$-period at most 2. In this article, we study examples of this
setup arising from finite, tame and wild $\CM$-type isolated
hypersurface singularities $R$. The endomorphism algebras of the
cluster tilting objects in the tame case occur in  lists in
\cite{BS,Er,Sk}. We also obtain a large class of symmetric finite
dimensional algebras where the stable AR-quiver consists only of tubes
of rank one or two. Examples of (wild) selfinjective algebras whose stable
AR-quiver consists only of tubes of rank one or three were known
previously \cite{AR}.


In the process we investigate the relationship between cluster tilting and maximal rigid objects.
It is of interest to know if the first property implies
the second one in general. In this paper we provide interesting
examples where this is not the case. The setting we deal with are the
simple isolated hypersurface singularities $R$ in dimension one over an
algebraically closed field $k$,
with the stable category $\underline{\CM}(R)$ of maximal
Cohen-Macaulay $R$-modules being our 2-CY category. These
singularities are indexed by the Dynkin diagrams, and in the cases $D_n$
for odd $n$ and $E_7$ we give examples of maximal rigid objects which are not cluster tilting.  
We also deal with cluster tilting and (maximal) rigid objects in the category $\CM(R)$, defined in an analogous way.

We also investigate the other Dynkin diagrams, and it is interesting
to notice that there are cases with no nonzero rigid objects
($A_n$, $n$ even, $E_6, E_8$), and cases where the maximal rigid objects coincide with
the cluster tilting objects ($A_n, n$ odd and $D_n, n$ even). In the
last case we see that both loops and 2-cycles can occur for the
associated 2-CY tilted algebras, whereas this never
happens for the cases $\mathcal{C}_H$ and $\underline{\mod}(\Lambda)$
\cite{BMRRT,BMR2,GLSc}. The results are also valid for any 
odd-dimensional simple hypersurface singularity, since the stable
categories of Cohen-Macaulay modules are all triangle equivalent \cite{Kn,S}.

We shall construct a large class of one-dimensional
hypersurface singularities $R$, where $\CM(R)$ or $\ul{\CM}(R)$ has a cluster tilting object,
including examples coming from simple singularities and minimally
elliptic singularities.
We classify all rigid objects in $\CM(R)$ for these $R$, in particular,
we give a bijection between cluster tilting objects in $\CM(R)$ and
elements in a symmetric group. Our method is based on a higher theory of
almost split sequences \cite{I1,I2}, and a crucial role is played by the
endomorphism algebras $\End_R(T)$
(called `three-dimensional Auslander algebras')
of cluster tilting objects $T$ in $\CM(R)$. These algebras
have global dimension three, and have 2-CY tilted algebras as stable factors.
The functor $\Hom_R(T,\ ):\CM(R)\to\mod(\End_R(T))$ sends cluster
tilting objects in $\CM(R)$ to tilting modules over $\End_R(T)$.
By comparing cluster tilting mutations in $\CM(R)$ and tilting mutation in
$\CM(\End_R(T))$, we can apply results on tilting mutation due to
Riedtmann-Schofield \cite{RS} and Happel-Unger \cite{HU1,HU2} to get information on cluster tilting
objects in $\CM(R)$.

We focus on the interplay between cluster tilting theory and
birational geometry (see section 5 for definitions). In \cite{V1,V2}, Van den Bergh established a
relationship between crepant resolutions of singularities and certain
algebras called non-commutative crepant resolutions, via derived equivalence. 
It is known that endomorphism algebras of cluster tilting objects of
three-dimensional normal Gorenstein singularities are 3-CY in the sense that the
bounded derived category of finite length modules is 3-CY, and 
they form a class of non-commutative crepant resolutions \cite{I2,IR}.
Thus we have a connection between cluster tilting theory and
birational geometry. We translate Katz's criterion \cite{Kat} for
three-dimensional $cA_n$--singularities for existence of crepant
resolutions to a criterion 
for one-dimensional hypersurface singularities for existence of cluster
tilting objects.
Consequently the class of hypersurface singularities,
which are shown to have cluster tilting objects by using higher almost split
sequences, are exactly the class having non-commutative crepant resolutions.
However we do not know whether the number of cluster tilting objects has a
meaning in birational geometry.

In section \ref{additive} we investigate maximal rigid objects and cluster tilting
objects in $\underline{\CM}(R)$  for simple one-dimensional
hypersurface singularities. We decide whether extension spaces are
zero or not by using covering techniques. In section \ref{sec2} we point out
that we could also use the computer program {\tt Singular} \cite{GP} to
accomplish the same thing. In section \ref{onedim} we construct cluster tilting
objects for a large class of isolated hypersurface singularities,
where the associated 2-CY tilted algebras can be of finite, tame or
wild representation type. We also classify cluster tilting and indecomposable rigid objects for this class.
In section \ref{sec3} we establish a connection between existence of cluster
tilting objects and existence of small resolutions. 
In section \ref{sec4} we give a geometric approach to some of the results in section \ref{onedim}.
Section \ref{CYtilted} is devoted to computing some concrete examples of 2-CY
tilted algebras. In section \ref{finite} we
generalize results from section \ref{additive} to 2-CY triangulated
categories with only a finite number of indecomposable objects. 


We refer to \cite{Y} as a general reference for representation theory
of Cohen-Macaulay rings, and \cite{AGV,GLSh} for classification of singularities.    

Our modules are usually right modules, and composition of maps $fg$ means first $g$, then $f$.
We call a module {\it basic} if it is a direct sum of mutually non-isomorphic indecomposable modules.

\section*{Acknowledgment} The first author would like to thank Duco van Straten and the second author would like to thank Atsushi Takahashi
and Hokuto Uehara for stimulating discussions.

\section{Main results}\label{main}

Let $(R,{\mathfrak m})$ be a local complete $d$-dimensional commutative noetherian
Gorenstein isolated singularity and $R/{\mathfrak m}=k\subset R$,
where $k$ is an algebraically closed field of characteristic zero.
We denote by $\CM(R)$ the category of maximal Cohen-Macaulay modules over $R$. Then $\CM(R)$ is a Frobenius category (i.e. an exact category with enough projectives and injectives which coincide),
and so the stable category $\ul{\CM}(R)$ is a $\Hom$-finite triangulated category with shift functor $\Sigma=\Omega^{-1}$ \cite{H}.
For an integer $n$, we say that $\ul{\CM}(R)$ or $\CM(R)$ is {\it $n$-CY} if there exists a functorial isomorphism
$$\ul{\Hom}_R(X,Y)\simeq D\ul{\Hom}_R(Y,\Sigma^nX)$$
for any $X,Y\in\CM(R)$.

We collect some fundamental results.
\begin{itemize}
  \item We have AR-duality 
    $$\ul{\Hom}_R(X,Y)\simeq D\Ext^1_R(Y,\tau X)$$
    with $\tau\simeq\Omega^{2-d}$ \cite{Au}. In particular, $\ul{\CM}(R)$ is $(d-1)$-CY.
  \item If $R$ is a hypersurface singularity, then $\Sigma^2= \id$ \cite{Ei}.
\newline Consequently, if $d$ is odd, then $\tau=\Omega$ and $\ul{\CM}(R)$ is 2-CY. If $d$ is even, then $\tau=\id$ and $\ul{\CM}(R)$ is 1-CY, hence 
any non-free Cohen-Macaulay $R$-module $M$ satisfies $\Ext^1_R(M,M)\neq0$.
  \item (Kn\"orrer periodicity) $$\ul{\CM}(k[[x_0,\cdots, x_d,y,z]]/{(f+yz)})\simeq \ul{\CM}(k[[x_0,\cdots, x_d]]/(f))$$
for any $f\in k[[x_0,\cdots,x_d]]$ \cite{Kn} (\cite{S} in characteristic two).
\end{itemize}

We state some of the definitions, valid more generally, in the context of $\CM(R)$ and $\ul{\CM}(R)$.

\begin{definition}
Let $\mathcal{C}=\CM(R)$ or $\ul{\CM}(R)$.
We call an object $M\in\mathcal{C}$
\begin{itemize}
\item \emph{rigid} if $\Ext^1_R(M,M)=0$,
\item \emph{maximal rigid} if it is rigid and any rigid
  $N\in\mathcal{C}$ satisfying $M\in\add N$ satisfies $N\in\add M$,
\item \emph{cluster tilting} if $\add M=\{X\in\mathcal{C}\ |\ \Ext^1_R(M,X)=0\}=\{X\in\mathcal{C}\ |\ \Ext^1_R(X,M)=0\}$.
\end{itemize}
\end{definition}

Cluster tilting objects are maximal rigid, but we show that the
converse does not necessarily hold for 2-CY triangulated categories $\ul{\CM}(R)$.
If $\mathcal{C}$ is 2-CY, then $M\in\mathcal{C}$ is cluster tilting
if and only if $\add M=\{X\in\mathcal{C}\ |\ \Ext^1_R(M,X)=0\}$.

\begin{definition}\label{cluster tilting mutation}
Let $\mathcal{C}=\CM(R)$ (or $\ul{\CM}(R)$) be 2-CY and
$M\in\mathcal{C}$ a basic cluster tilting object.
Take an indecomposable summand $X$ of $M=X\oplus N$.
Then there exist short exact sequences (or triangles) (called \emph{exchange
  sequences})
$$X\stackrel{g_1}{\to}N_1\stackrel{f_1}{\to}Y\ \mbox{ and }\ Y\stackrel{g_0}{\to}N_0\stackrel{f_0}{\to}X$$
such that $N_i\in\add N$ and $f_0$ is a minimal right $(\add N)$-approximation.
Then $Y\oplus N$ is a basic cluster tilting object again called \emph{cluster tilting mutation}
of $M$ \cite{BMRRT,GLSc}\cite[Def. 2.5, Th. 5.3]{IY}.
In this case $f_1$ is a minimal right $(\add N)$-approximation
and $g_i$ is a minimal left $(\add N)$-approximation automatically,
so $X\oplus N$ is a cluster tilting mutation of $Y\oplus N$.
It is known that there are no more basic cluster tilting objects
containing $N$ \cite[Th. 5.3]{IY}.
\end{definition}

\bigskip
Let $R=k[[x,y,z_2,\cdots, z_d]]/(f)$ be a simple hypersurface
    singularity so that in characteristic zero $f$ is one of the following polynomials,
\[\begin{array}{ccc}
(A_n)&x^2+y^{n+1}+z_2^2+z_3^2+\cdots+z_d^2&(n\ge1)\\
(D_n)&x^2y+y^{n-1}+z_2^2+z_3^2+\cdots+z_d^2&(n\ge4)\\
(E_6)&x^3+y^4+z_2^2+z_3^2+\cdots+z_d^2&\\
(E_7)&x^3+xy^3+z_2^2+z_3^2+\cdots+z_d^2&\\
(E_8)&x^3+y^5+z_2^2+z_3^2+\cdots+z_d^2&
\end{array}\]
Then $R$ is of finite Cohen-Macaulay representation type \cite{Ar,GK,Kn,S}.

We shall show the following result in section \ref{additive} using additive
functions on the AR quiver. We shall explain another proof using
{\tt Singular} in section \ref{sec2}.
\begin{theorem}\label{main1}
Let $R$ be a simple hypersurface singularity of dimension $d\ge1$
over an algebraically closed field $k$ of characteristic zero.

(1) Assume that $d$ is even. Then $\underline{\rm CM}(R)$ does not
have non-zero rigid objects.

(2) Assume that $d$ is odd. Then the number of indecomposable rigid
objects, basic cluster tilting objects, basic maximal rigid objects, and
indecomposable summands of basic maximal rigid objects in $\underline{\rm CM} (R)$ are as follows:
\[\begin{array}{|c|c|c|c|c|}\hline
f&\mbox{indec. rigid}&\mbox{cluster tilting}&\mbox{max.
  rigid}&\mbox{summands of max. rigid}\\ \hline\hline
(A_n)\ n:\mbox{odd}&2&2&2&1\\ \hline
(A_n)\ n:\mbox{even}&0&0&1&0\\ \hline
(D_n)\ n:\mbox{odd}&2&0&2&1\\ \hline
(D_n)\ n:\mbox{even}&6&6&6&2\\ \hline
(E_6)&0&0&1&0\\ \hline
(E_7)&2&0&2&1\\ \hline
(E_8)&0&0&1&0\\ \hline
\end{array}\]
\end{theorem}

\bigskip
We also consider a {\it minimally elliptic} curve singularity $T_{p,q}(\lambda)$ ($p\le q$).
Assume for simplicity that our base field $k$ is algebraically closed
of characteristic zero. Then these singularities are given by the
equations
$$
x^p + y^q + \lambda x^2 y^2 = 0,
$$
where $\frac{1}{p} + \frac{1}{q} \le \frac{1}{2}$ and certain values of $\lambda \in k$ have to be excluded. 
They are of tame Cohen-Macaulay representation type \cite{Di,Ka,DG}.
We divide into two cases.

(i) Assume $\frac{1}{p} + \frac{1}{q} = \frac{1}{2}$. This case occurs if
and only if $(p,q)=(3,6)$ or $(4,4)$, and $T_{p,q}(\lambda)$ is called
{\it simply elliptic}.
The corresponding coordinate rings can be written in the form 
$$
T_{3,6}(\lambda) = k[[x,y]]/(y(y-x^2)(y-\lambda x^2))
$$
and 
$$
T_{4,4}(\lambda) = k[[x,y]]/(xy(x-y)(x -\lambda y)), 
$$
where in both cases $\lambda \in k \setminus\{0,1\}$. 

(ii) Assume $\frac{1}{p} + \frac{1}{q} < \frac{1}{2}$.
Then $T_{p,q}(\lambda)$ does not depend on the continuous
parameter $\lambda$, and is called a \emph{cusp} singularity. In this
case 
the corresponding coordinate rings can be written in the form 
$$
T_{p,q} = k[[x,y]]/((x^{p-2}-y^2)(x^2-y^{q-2})).
$$

We shall show the following result in section \ref{sec4} by applying a result
in birational geometry.
\begin{theorem}\label{main2}
Let $R$ be a minimally elliptic curve singularity $T_{p,q}(\lambda)$
over an algebraically closed field $k$ of characteristic zero.
\begin{itemize}
\item[(a)] $\ul{\CM}(R)$ has a cluster tilting object if and only if $p=3$ and $q$ is
even or if both $p$ and $q$ are even.
\item[(b)] The number of indecomposable
rigid objects, basic cluster tilting objects, and indecomposable summands of
basic cluster tilting objects in $\ul{\CM}(R)$ are as follows:
\[\begin{array}{|c|c|c|c|}\hline
p,q&\mbox{indec. rigid}&\mbox{cluster tilting}&\mbox{summands of
  cluster tilting}\\ \hline\hline
p=3,\ q:\mbox{even}&6&6&2\\ \hline
p,q:\mbox{even}&14&24&3\\ \hline
\end{array}\]
\end{itemize}
\end{theorem}

\bigskip
We also prove the following general theorem, which includes 
both Theorem \ref{main1} (except the assertion on maximal rigid
objects) and Theorem \ref{main2}.
The `if' part in (a) and the assertion (b)
are proved in section \ref{onedim} by a purely homological method.
The proof of (a), including another proof of the `if' part, 
is given in section \ref{sec4} by applying Katz's criterion in
birational geometry. 

\begin{theorem}\label{main3}
Let $R=k[[x,y]]/(f)$ ($f\in(x,y)$) be a one-dimensional reduced hypersurface singularity
over an algebraically closed field $k$ of characteristic zero.
\begin{itemize}
\item[(a)] $\ul{\CM}(R)$ has a cluster tilting object
if and only if $f$ is a product $f=f_1\cdots f_n$ with $f_i\notin(x,y)^2$.
\item[(b)] The number of indecomposable
rigid objects, basic cluster tilting objects, and indecomposable summands of
basic cluster tilting objects in $\ul{\CM}(R)$ are as follows:
\[\begin{array}{|c|c|c|}\hline
\mbox{indec. rigid}&\mbox{cluster tilting}&\mbox{summands of
  cluster tilting}\\ \hline\hline
2^n-2&n!&n-1\\ \hline
\end{array}\]
\end{itemize}
\end{theorem}

\bigskip
The following result gives a bridge between cluster tilting theory and
birational geometry. The terminologies are explained in section \ref{sec3}.
\begin{theorem}\label{main4}
Let $(R,{\mathfrak m})$ be a three-dimensional isolated $cA_n$--singularity 
over an algebraically closed field $k$ of characteristic zero
defined by the equation $g(x,y)+zt$ and $R'$ a one-dimensional singularity defined by $g(x,y)$.
Then the following conditions are equivalent.
\begin{itemize}
\item[(a)] $\Spec(R)$ has a small resolution.
\item[(b)] $\Spec(R)$ has a crepant resolution.
\item[(c)] $(R,{\mathfrak m})$ has a non-commutative crepant resolution.
\item[(d)] $\ul{\CM}(R)$ has a cluster tilting object.
\item[(e)] $\ul{\CM}(R')$ has a cluster tilting object.
\item[(f)] The number of irreducible power series in the prime decomposition of $g(x,y)$ is $n+1$.
\end{itemize}
\end{theorem}

\bigskip
We end this section by giving an application to finite dimensional algebras.
A {\it 2-CY tilted algebra} is an endomorphism ring $\End_{\mathcal{C}}(M)$
of a cluster tilting object $T$ in a 2-CY triangulated category $\mathcal{C}$.
In section \ref{CYtilted}, we shall show the following result and compute
2-CY tilted algebras associated with minimally elliptic curve singularities.

\begin{theorem}\label{main5}
  Let $(R,{\mathfrak m})$ be an odd-dimensional isolated hypersurface singularity
and $\ga$ a 2-CY tilted algebra coming from $\ul{\CM}(R)$. Then we have the following.
\begin{itemize}
  \item[(a)] $\ga$ is a symmetric algebra.
  \item[(b)] All components in the stable AR-quiver of infinite type $\ga$ are tubes of rank 1 or 2.
\end{itemize}
\end{theorem}

For example, put 
$$R=k[[x,y]]/((x-\lambda_1y)\cdots(x-\lambda_ny))\ \mbox{ and }\ 
M=\bigoplus_{i=1}^nk[[x,y]]/((x-\lambda_1y)\cdots(x-\lambda_iy))$$
for distinct elements $\lambda_i\in k$.
Then $M$ is a cluster tilting object in $\CM(R)$ by Theorem \ref{XAA},
so $\ga=\ul{\End}_R(M)$ satisfies the conditions in Theorem \ref{main5}.
Since $\CM(R)$ has wild Cohen-Macaulay representation type if $n>4$
\cite[Th. 3]{DG}, we should get a
family of examples of finite dimensional symmetric $k$-algebras
whose stable AR-quiver consists only of tubes of rank 1 or 2, and are of
wild representation type.

\section{Simple hypersurface singularities}\label{additive}

Let $R$ be a one-dimensional  simple hypersurface singularity. In
this case the AR-quivers are known for $\CM(R)$ \cite{DW}, and so also
for $\underline{\CM}(R)$. We use the notation from \cite{Y}. 

In order to locate the indecomposable rigid modules $M$, that is, the modules $M$ with \newline$\Ext^1(M,M)=0$, the following lemmas are useful, where part (a) of the first one is proved in \cite{HKR}, and the second one is a direct consequence of \cite{KR} (generalizing
\cite{BMR1}).
\begin{lemma}\label{L1.1}
  \begin{itemize}
  \item[(a)] Let $\mathcal{C}$ be an abelian or triangulated $k$-category with finite dimensional homomorphism spaces. Let $A\xrightarrow{
\left( 
\begin{smallmatrix}
  f_1\\
  f_2
\end{smallmatrix}
\right)} B_1\oplus B_2\xrightarrow{(g_1,g_2)} C $ be a short exact sequence or a triangle, where $A$ is  indecomposable, $B_1$ and $B_2$ nonzero, and $(g_1,g_2)$ has no nonzero indecomposable summand which is an isomorphism. Then
    $\Hom(A,C) \neq 0$ .
  \item[(b)]Let $0\to A\xrightarrow{f} B\xrightarrow{g} C\to 0$ be an almost split sequence in $\CM(R)$, where $R$ is an isolated hypersurface singularity, and $B$ has at least two indecomposable nonprojective summands in a decomposition of $B$ into a direct sum of indecomposable modules. Then $\Ext^1(C,C)\neq0$.
  \end{itemize}
\end{lemma}
\begin{proof}
  (a) See \cite[Lem. 6.5]{HKR}.

(b) Using (a) together with the above AR-formula and $\tau^2=\id$, we obtain  $ D\Ext^1(C,C)\simeq$

\noindent
$\underline{\Hom}(\tau^{-1}C,C)=\underline{\Hom}(\tau C,C)\simeq\underline{\Hom}(A,C)\neq 0$, where $ D=\Hom_k(\text{ },k)$.
\end{proof}

\begin{lemma}\label{L1.2}
  Let $T$ be  a cluster tilting object in the $\Hom$-finite connected 2-{\rm CY} category
  $\mathcal{C}$, and $\Gamma=\End_{\cc}(T)$.
\begin{itemize}
\item[(a)] The functor $G=\Hom_{\cc}(T,\text{ })\colon\cc\to \mod(\Gamma)$ induces an equivalence of categories \newline$\ol{G}\colon\cc/{\add(\tau T)}\to \mod(\Gamma)$.
\item[(b)] The AR-quiver for $\Gamma$ is as a translation quiver obtained from the AR-quiver for $\cc$ by removing the vertices corresponding to the indecomposable summands of $\tau T$.
\item[(c)] Assume $\tau^2=\id$. Then  we have the following.
\begin{itemize}
  \item[(i)] $\Gamma$ is a symmetric algebra.
  \item[(ii)] The indecomposable nonprojective $\Gamma$-modules have $\tau$-period one or two.
\item[(iii)] If $\cc$ has an infinite number of nonisomorphic indecomposable objects, then all components in the stable AR-quiver of $\Gamma$ are tubes of rank one or two.
\end{itemize}
\item[(d)] If $\cc$ has only a finite number $n$ of nonisomorphic
  indecomposable objects, and $T$ has $t$ nonisomorphic indecomposable
  summands, then there are $n-t$ nonisomorphic indecomposable
  $\Gamma$-modules. 
\end{itemize}
\end{lemma}

\begin{proof}
For (a) and (b) see  \cite{BMR1,KR}. Since $\cc$ is 2-CY, we have $\tau=\Sigma$, and a functorial isomorphism 
$$D\Hom_{\cc}(T,T)\simeq\Hom_{\cc}(T,\Sigma^2T)=\Hom_{\cc}(T,\tau^2T)\simeq\Hom_{\cc}(T,T).$$
This shows that $\Gamma$ is symmetric. Let $C$ be an indecomposable
nonprojective $\Gamma$-module. Viewing $C$ as an object in $\cc$ we have $\tau^2_{\cc}C\simeq C$, and $\tau C$ is not a projective $\ga$-module since $C$ is not removed. Hence we have $\tau^2_{\ga}C\simeq C$. If $\cc$ has an infinite number of nonisomorphic indecomposable objects, then $\ga$ is of infinite type. Then each component of the AR-quiver is infinite, and hence is a tube of rank one or two. Finally, (d) is a direct consequence of (a).
\end{proof}

We also  use that in our cases we have a covering functor $\Pi\colon k(\mathbb{Z}Q)\to \underline{\CM}(R)$, where $Q$ is the appropriate Dynkin quiver and $k(\mathbb{Z}Q)$ is the mesh category of the translation quiver $\mathbb{Z}Q$ \cite{Rit,Am}, (see also \cite[Section 4.4]{I1} for another explanation using functorial methods).

For the one-dimensional simple hypersurface singularities we have the
cases $A_n$ ($n$ even or odd), $D_n$ ($n$ odd or even), $E_6$, $E_7$ and $E_8$. We now investigate them case by case. 
\begin{proposition}\label{p1.3}
  In the case $A_n$ (with $n$ even) there are no indecomposable rigid objects.
\end{proposition}
\begin{proof}
  We have the stable AR-quiver
$$\xymatrix@C0.5cm@R0.5cm{
    &I_1\ar@(dl,dr)@{--}\ar@<0.5ex>[r]& I_2\ar@(dl,dr)@{--}\ar@<0.4ex>[l]\ar@<0.4ex>[r]& \cdots\ar@<0.4ex>[r]\ar@<0.4ex>[l]& I_{n/2}\ar@(dl,dr)@{--}\ar@<0.4ex>[l]
}$$\\

Here, and later, a dotted line between two indecomposable modules
means that they are connected via $\tau$.

Since $\tau I_j\simeq I_j $ for  each $j$, $\Ext^1(I_j,I_j)\neq 0$ for $j= 1,\cdots, n/2$. Hence no $I_j$ is rigid.
\end{proof}
\begin{proposition}\label{p1.4}
  In the case $A_n$ (with $n$ odd) the maximal rigid objects coincide
  with the cluster tilting objects. There are two indecomposable ones, and the corresponding 2-CY tilted algebras are $k[x]/(x^{\frac{(n+1)}{2}})$. 
\end{proposition}
\begin{proof}
For simplicity, we write $l=({n-1})/2$. We have the stable AR-quiver 
$$\xymatrix@C=0.5cm@R=0.5cm{
    &&&& N_{-}\ar@{--}[dd]\ar@<0.5ex>[dl]\\
    M_1\ar@(dl,dr)@{--}\ar@<0.5ex>[r]& M_2\ar@(dl,dr)@{--}\ar@<0.5ex>[r]\ar@<0.5ex>[l]&
    \cdots\ar@<0.5ex>[r]\ar@<0.5ex>[l]&
    M_{l}\ar@(dl,dr)@{--}\ar@<0.5ex>[ur]\ar@<0.5ex>[l]\ar@<0.5ex>[dr]\\ 
    &&&& N_{+}\ar@<0.5ex>[ul]
  }$$
Since $\tau M_i\simeq M_i$ for $i=1, \cdots ,l$, we have
$$\Ext^1(M_i,M_i)\simeq \underline{\Hom}(M_i,\tau M_i)\simeq
\underline{\Hom}(M_i,M_i)\neq 0.$$ So only the indecomposable objects
$N_{-}$ and $N_{+}$ could be rigid. We use covering techniques and additive functions to compute the
support of $\underline{\Hom}(N_-,\ )$, where we refer to \cite{BG} for the meaning of the diagrams below.

$
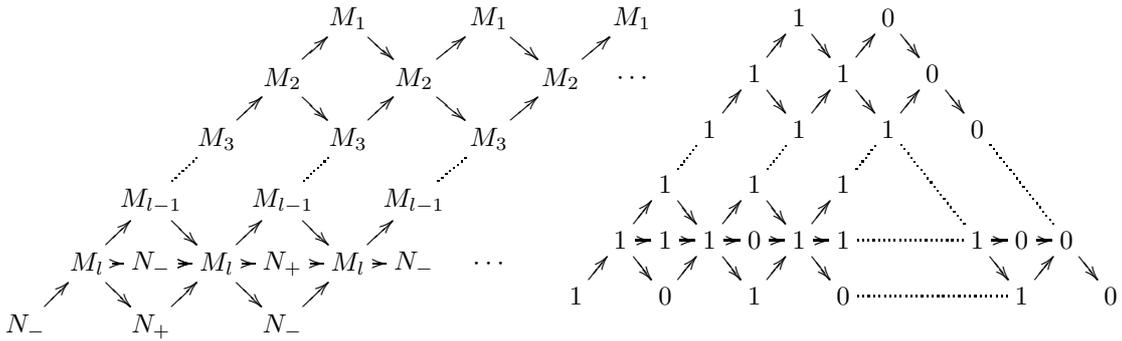
\xymatrix@C=0cm@R0.3cm{
  &&&&& M_1\ar[dr]&&M_1\ar[dr]&&\ M_1\ \\
  &&&&M_2\ar[ur]\ar[dr]&& M_2\ar[ur]\ar[dr]&&\ M_2\ \ar[ur]&\cdots\\
  &&& M_3\ar[ur]&&M_3\ar[ur]&&\ M_3\ar[ur]\ &&\\
  &&M_{l-1}\ar[dr]\ar@{..}[ur]&&M_{l-1}\ar[dr]\ar@{..}[ur]&&
  M_{l-1}\ar@{..}[ur]&&&\\
  &M_l\ar[ur]\ar[dr]\ar[r]&N_-\ar[r]\ar[r]&
  M_l\ar[ur]\ar[dr]\ar[r]&N_+\ar[r]&M_l\ar[ur]\ar[r]&N_-&\cdots&&\\
  \ N_{-}\ar[ur]\ && N_{+}\ar[ur]&& N_{-}\ar[ur]&&&&&
}$\vspace{0.1cm}\strut\kern-4em
$\xymatrix@C0.2cm@R0.3cm{
  &&&&& 1\ar[dr]&& 0\ar[dr]&&&&&\\
  &&&& 1\ar[ur]\ar[dr]&& 1\ar[ur]\ar[dr]&& 0\ar[dr]&&&&\\
  &&& 1\ar[ur]&& 1\ar[ur]&& 1\ar@{..}[drdr]\ar[ur]&&
  0\ar@{..}[drdr]&&&\\ 
 &&1\ar@{..}[ur]\ar[dr]&& 1\ar@{..}[ur]\ar[dr]&&1\ar@{..}[ur]&&&&&&\\ 
 & 1\ar[ur]\ar[dr]\ar[r]&1\ar[r]& 1\ar[ur]\ar[dr]\ar[r]&0\ar[r]&
  1\ar[ur]\ar[dr]\ar[r]&1\ar@{..}[rrr]&&&1\ar[dr]\ar[r]&0\ar[r]&0\ar[dr]&\\
 1\ar[ur]&& 0\ar[ur]&& 1\ar[ur]&& 0\ar@{..}[rrrr]&&&& 1\ar[ur]&& 0
}$\\

We see that $\underline{\Hom}(N_-,N_+)=0$, so
$\Ext^1(N_+,N_+)=\Ext^1(N_+,\tau N_-)=0$, and
$\Ext^1(N_-,N_-)=0$. Since $\Ext^1(N_+,N_-)\neq 0$, we see that $N_+$
and $N_-$  are exactly the maximal rigid objects. Further
$\underline{\Hom}(N_-,M_i)\neq 0$ for all $i$, so $\Ext^1(N_+,M_i)\neq
0$ and $\Ext^1(N_-,M_i)\neq 0$ for all $i$. This shows that $N_+$ and
$N_-$ are also cluster tilting objects. 

The description of the cluster tilted algebras follows directly from the above picture.
\end{proof}
\begin{proposition}\label{p1.5}
  In the case $D_n$
with $n$ odd we have two maximal rigid objects,
  which both are indecomposable, and neither one is cluster tilting.
\end{proposition}
\begin{proof}
  We have the AR-quiver
  $$
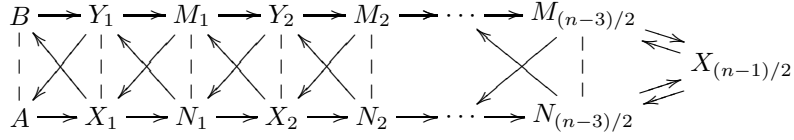
\xymatrix@R0.1cm@C0.5cm{
    B\ar@{--}[dd]\ar[r]& Y_1\ar[ddl]\ar[r]\ar@{--}[dd]&
    M_1\ar[ddl]\ar@{--}[dd]\ar[r]& Y_2\ar[r]\ar[ddl]\ar@{--}[dd]& 
    M_2\ar[r]\ar[ddl]\ar@{--}[dd]\ar[r]& 
    \cdots\ar[r]& M_{{(n-3)}/2}\ar@<0.5ex>[dr]\ar[ddl]\ar@{--}[dd]&\\
    &&&&&&& X_{{(n-1)}/2}\ar@<0.5ex>[ul]\ar@<0.5ex>[dl]\\
    A\ar[r]& X_1\ar[uul]\ar[r]& N_1\ar[uul]\ar[r]& X_2\ar[uul]\ar[r]&
    N_2\ar[uul]\ar[r]& \cdots\ar[r]& N_{{(n-3)}/2}\ar[uul]\ar@<0.5ex>[ur]&
}$$
  Using Lemma \ref{L1.1}, the only candidates for being indecomposable
  rigid are $A$ and $B$. We compute the support of
  $\underline{\Hom}(A,\ )$
  
  $
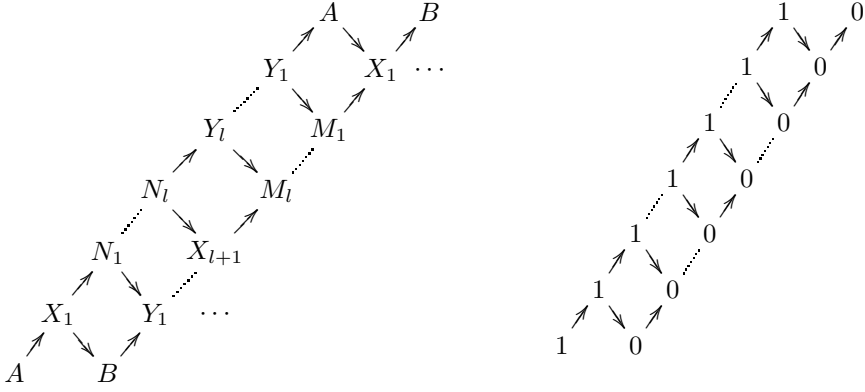
\xymatrix@C0cm@R0.3cm{
    &&&&&& A\ar[dr]&&B\\
    &&&&&Y_1\ar[ur]\ar[dr]&& X_1\ar[ur]&\cdots\\
    &&&&Y_l\ar@{..}[ur]\ar[dr]&& M_1\ar[ur]&\\
    &&&N_l\ar[ur]\ar[dr]&&M_l\ar@{..}[ur]&&\\
    &&N_1\ar@{..}[ur]\ar[dr]&&X_{l+1}\ar[ur]&&&\\
    &X_1\ar[ur]\ar[dr]&& Y_1\ar@{..}[ur]&\cdots&&&\\
    A\ar[ur]&& B\ar[ur]&&&&&
  }$ \hspace{1cm}
  $\xymatrix@C0.1cm@R0.3cm{
    &&&&&& 1\ar[dr]&&0 \\
    &&&&&1\ar[ur]\ar[dr]&& 0\ar[ur]&\\
    &&&&1\ar@{..}[ur]\ar[dr]&& 0\ar[ur]&&\\
    &&&1\ar[ur]\ar[dr]&& 0\ar@{..}[ur]&&&\\
    &&1\ar@{..}[ur]\ar[dr]&& 0\ar[ur]&&&&\\
    &1\ar[ur]\ar[dr]&& 0\ar@{..}[ur]&&&&&\\
    1\ar[ur]&& 0\ar[ur]&&&&&&
  }$

\noindent
  where $B=\tau A$ and $l={(n-3)}/2$. We see that
  $\underline{\Hom}(A,B)=0$, so that $\Ext^1(A,A)=0$. Then $A$ is
  clearly maximal rigid. Since $\underline{\Hom}(A,M_1)=0$, we have
  $\Ext^1(A,N_1)=0$, so $A$ is not cluster tilting. Alternatively, 
  we could use that we see that $\ul{\End}(A)^{\op}\simeq k[x]/(x^2)$,
  which has two indecomposable modules, whereas
  $\underline{\CM}(R)$ has $2n-3$ indecomposable objects. If  $A$
  was cluster tilting, $\ul{\End}(A)^{\op}$ would have had $2n-3-1=2n-4$
  indecomposable modules, by Lemma \ref{L1.2}.
\end{proof}
\begin{proposition}\label{p1.6}
  In the case $D_{2n}$
with $n$ a positive integer we have that the maximal rigid objects
  coincide with the cluster tilting ones. There are 6 of them, and
  each is a direct sum of two nonisomorphic indecomposable objects.

  The corresponding 2-CY-tilted algebras are given by the quiver with
  relations 
  $\xymatrix@C0.5cm@R0.1cm{
    \cdot\ar@<0.4ex>[r]^{\alpha}& \cdot\ar@<0.4ex>[l]^{\beta}\\
    }$
  $\alpha\beta\alpha =0= \beta\alpha\beta$ in the case $D_4$, and by 
  $\xymatrix@C0.5cm@R0.1cm{
    \cdot \ar@(ul,dl)[]_{\gamma} \ar@<0.4ex>[r]^{\alpha}&
    \cdot\ar@<0.4ex>[l]^{\beta}\\ 
    }$
with $\gamma^{n-1}=\beta\alpha$, $\gamma\beta=0=\alpha\gamma$
and 
$\xymatrix@C0.5cm@R0.1cm{
    \cdot\ar@<0.4ex>[r]^{\alpha}& \cdot\ar@<0.4ex>[l]^{\beta}\\
    }$
with 
$(\alpha\beta)^{n-1}\alpha=0=(\beta\alpha)^{n-1}\beta$ for  $2n>4$.
\end{proposition}
\begin{proof}
  We have the AR-quiver
  $$
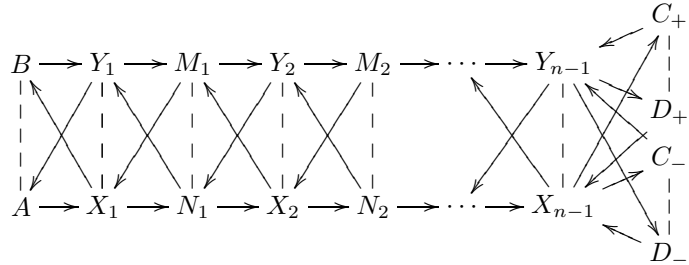
\xymatrix@R0.1cm@C0.5cm{
    &&&&&&& C_+\ar[dl]\ar@{--}[dd]\\
    B\ar@{--}[ddd]\ar[r]& Y_1\ar@{--}[ddd]\ar[dddl]\ar[r]\ar@{--}[ddd]&
    M_1\ar[r]\ar[dddl]\ar@{--}[ddd]& Y_2\ar[r]\ar[dddl]\ar@{--}[ddd]&
    M_2\ar[r]\ar[dddl]\ar@{--}[ddd]& \cdots\ar[r]&
    Y_{n-1}\ar[dr]\ar[ddddr]\ar@{--}[ddd]\ar[dddl]&\\
    &&&&&&& D_+\ar[ddl]\\
    &&&&&&&C_-\ar[uul]\ar@{--}[dd]\\
    A\ar[r]& X_1\ar[r]\ar[uuul]& N_1\ar[uuul]\ar[r]& X_2\ar[uuul]\ar[r]&
    N_2\ar[uuul]\ar[r]& \cdots\ar[r]&
    X_{n-1}\ar[uuul]\ar[uuuur]\ar[ur]\\
    &&&&&&& D_-\ar[ul]
  }$$
  By Lemma \ref{L1.1}, the only possible indecomposable rigid objects
  are: $A$, $B$, $C_+$, $C_-$, $D_+$, $D_-$.
  
  We compute the support of $\underline{\Hom}(C_+,\ )$:
  $$\xymatrix@C0cm@R0.3cm{
    && D_{-}\ar[dr]&& C_{-}\ar[dr]&&&&&&&& D_{-}\ar[dr]&& C_{-}\ar[dr]&&D_-\\
    C_{+}\ar[r]& Y_l\ar[dr]\ar[ur]\ar[r]&D_+\ar[r]&
    X_l\ar[dr]\ar[ur]\ar[r]&C_+\ar[r]&Y_l\ar@{..}[rrrrrr]\ar[dr]
    &&&&&& Y_l\ar[rd]\ar[ur]\ar[r]&D_+\ar[r]& X_l\ar[ur]\ar[dr]\ar[r]&C_+\ar[r]&Y_l\ar[ur]
    &\cdots\\ 
    && N_{l-1}\ar[ur]\ar@{..}[drdr]&&
    M_{l-1}\ar[ur]\ar@{..}[drdr]&&N_{l-1}\ar@{..}[drdr]&&&&M_{l-1}\ar[ur]&& N_{l-1}\ar[ur]
    &&M_{l-1}\ar[ur]\\
    &&&&&&&&&&&&&&\\
    &&&& Y_2\ar[dr]&& X_2\ar[dr]&&Y_2\ar[dr]\ar@{..}[urur]&& X_2\ar@{..}[urur]\ar[dr]&&
    Y_2\ar@{..}[urur]&&\\
    &&&&& N_1\ar[ur]\ar[dr]&& M_1\ar[dr]\ar[ur]&& N_1\ar[ur]\ar[dr]&&M_1\ar[ur]&&&\\
    &&&&&& Y_1\ar[ur]\ar[dr]&& X_1\ar[dr]\ar[ur]&&Y_1\ar[ur]&\cdots&&&\\
    &&&&&&& A\ar[ur]&& B\ar[ur]&&&&&
  }$$
  where $l=n-1$
  $$\xymatrix@C0.2cm@R0.3cm{
    && 1\ar[dr]&& 0\ar[dr]&&&&&&&&1\ar[dr]&& 0\ar[dr]&& 0\\
    1\ar[r]& 1\ar[rd]\ar[ur]\ar[r]&0\ar[r]& 1\ar[rd]\ar[ur]\ar[r]&1\ar[r]&
    1\ar[dr]\ar@{..}[rrrrrr]&&&&&&1\ar[ur]\ar[dr]\ar[r]&0\ar[r]& 1\ar[rd]\ar[ur]\ar[r]
    &1\ar[r]& 0\ar[ur]&\\    
    && 1\ar[ur]\ar@{..}[drdr]&& 1\ar[ur]\ar@{..}[drdr]&&1\ar@{..}[drdr]&&&&1\ar[ur]
    &&1\ar[ur]&& 0\ar[ur]&&\\
    &&&&&&&&&&&&&&&&&&&&\\
    &&&& 1\ar[dr]&& 1\ar[dr]&&1\ar[dr]\ar@{..}[urur]&&
    1\ar[dr]\ar@{..}[urur]&&0\ar@{..}[urur]&&&\\
    &&&&& 1\ar[ur]\ar[dr]&& 1\ar[dr]\ar[ur]&&1\ar[dr]\ar[ur]&&0\ar[ur]&&&&\\ 
    &&&&&& 1\ar[ur]\ar[dr]&& 1\ar[dr]\ar[ur]&& 0\ar[ur]&&&&&\\
    &&&&&&& 1\ar[ur]&& 0\ar[ur]&&&&&&
  }$$
  We see that $\underline{\Hom}(C_+,D_+) =0$, so
  $\Ext^1(C_+,C_+)=0=\Ext^1(D_+,D_+)$. Further,
    $\underline{\Hom}(C_+,C_-)=0$, so 
  $\Ext^1(C_+,D_-)=0$. By symmetry $\Ext^1(D_-,D_-)=0=\Ext^1(C_-,C_-)$
  and $\Ext^1(D_+,C_-)=0$. Also $\Ext^1(C_+,A)=0$,
  $\Ext^1(C_+,B)\neq 0$, so $\Ext^1(D_+,B)=0$, $\Ext^1(D_+,A)\neq
  0$. Further $\Ext^1(C_+,X)\neq 0$ for $X\neq A, D_-, C_+$.
  
  We now compute the support of $\underline{\Hom}(A,\ )$
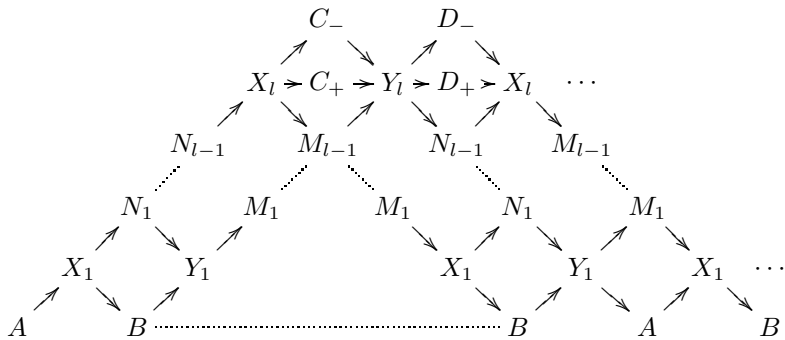
  $$\xymatrix@C0cm@R0.3cm{
    &&&&& C_-\ar[dr]&& D_-\ar[dr]&&&\\
    &&&& X_l\ar[ur]\ar[r]\ar[dr]& C_+\ar[r]& Y_l\ar[ur]\ar[r]\ar[dr]&
    D_+\ar[r]& X_l\ar[dr]&\cdots&\\
    &&& N_{l-1}\ar[ur]&& M_{l-1}\ar[ur]\ar@{..}[dr]&&
    N_{l-1}\ar[ur]\ar@{..}[dr]&&M_{l-1}\ar@{..}[dr]&\\
    && N_1\ar@{..}[ur]\ar[dr]&& M_1\ar@{..}[ur]&&M_1\ar[dr]&& N_1\ar[dr]&&M_1\ar[dr]\\
    &\ X_1\ar[ur]\ar[dr]\ && Y_1\ar[ur]&&&&X_1\ar[ur]\ar[dr]&& Y_1\ar[dr]\ar[ur]
    &&\ X_1\ar[dr]\ &\cdots\\
    \ A\ar[ur]\ && B\ar[ur]\ar@{..}[rrrrrr]&&&&&&B\ar[ur]&& A\ar[ur]&&\ B\ 
  }$$ 
  where $l=n-1$ and we have an odd number of columns and rows.
  $$
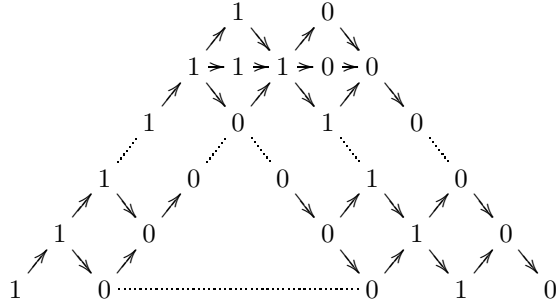
\xymatrix@C0.2cm@R0.3cm{
    &&&&& 1\ar[dr]&&0\ar[dr]&&&\\
    &&&& 1\ar[ur]\ar[r]\ar[dr]& 1\ar[r]& 1\ar[dr]\ar[r]\ar[ur]&0\ar[r]&0\ar[dr]&&\\
    &&& 1\ar[ur]&& 0\ar[ur]\ar@{..}[dr]&& 1\ar@{..}[dr]\ar[ur]&&0\ar@{..}[dr]&\\
    && 1\ar@{..}[ur]\ar[dr]&& 0\ar@{..}[ur]&&0\ar[dr]&& 1\ar[dr]&&0\ar[dr]\\
    & 1\ar[ur]\ar[dr]&& 0\ar[ur]&&&&0\ar[ur]\ar[dr]&& 1\ar[dr]\ar[ur]&&0\ar[dr]\\
    1\ar[ur]&& 0\ar[ur]\ar@{..}[rrrrrr]&&&&&&0\ar[ur]&& 1\ar[ur]&&0
  }$$
  We see that $\underline{\Hom}(A,B)=0$, so $\Ext^1(A,A)=0$, hence also
  $\Ext^1(B,B)=0$. Since $\underline{\Hom}(A,D_-)=0$, we have
  $\Ext^1(A,C_-)=0$, hence $\Ext^1(B,D_-)=0$. Since
  $\underline{\Hom}(A,C_-)\neq 0$, we have $\Ext^1(A,D_-)\neq 0$, so
  $\Ext^1(B,D_+)\neq 0$.
  
  It follows that $C_+\oplus D_-$, $C_-\oplus D_+$, $C_+\oplus A$,
  $D_+\oplus B$, $A\oplus C_-$ and $B\oplus D_-$ are maximal rigid.
  
  These are also cluster tilting: We have $\underline{\Hom}(A,X_i)\neq 0$, $\underline{\Hom}(A,N_i)\neq
  0$, so $\Ext^1(B,X_i)\neq 0$, $\Ext^1(B,N_i)\neq 0$. Similarly,
  $\Ext^1(A,Y_i)\neq 0$, $\Ext^1(A,M_i)\neq 0$. Also
  $\underline{\Hom}(C_+,Y_i)\neq 0$, $\underline{\Hom}(C_+,N_i)\neq 0$, so
  $\Ext^1(D_+,Y_i)\neq 0$, $\Ext^1(D_+,N_i)\neq 0$. Hence
  $\Ext^1(C_+,X_i)\neq 0$, \\$\Ext^1(C_+,M_i)\neq 0$. So
  $\Ext^1(D_-,Y_i)\neq 0$, $\Ext^1(D_-,N_i)\neq 0$, $\Ext^1(C_-,X_i)\neq
  0$, $\Ext^1(C_-,M_i)\neq 0$. We see that each indecomposable rigid object can be extended to a cluster tilting object in exactly two ways, which we would know from a general result in \cite[Th. 5.3]{IY}.

The exchange graph is as follows:
$$
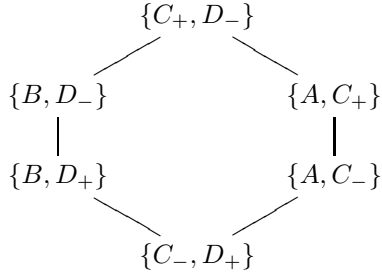
\xymatrix@C0.2cm@R0.5cm{
  &\{C_+,D_-\}\ar@{-}[dr]&\\
  \{B,D_-\}\ar@{-}[ur]\ar@{-}[d]&& \{A,C_+\}\ar@{-}[d]\\
  \{B,D_+\}\ar@{-}[dr]&& \{A,C_-\}\\
  & \{C_-,D_+\}\ar@{-}[ur]& 
}$$
Considering the above pictures, we get the desired description of the
corresponding 2-CY tilted algebras in terms of quivers with
relations.
\end{proof}
\begin{proposition}\label{P1.7}
  In the case $E_6$ there are no indecomposable rigid objects.
\end{proposition}
\begin{proof}
  We have the AR-quiver
  $$
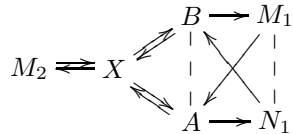
\xymatrix@C0.5cm@R0.2cm{
    && B\ar@<0.5ex>[dl]\ar@{--}[dd]\ar[r]& M_1\ar@{--}[dd]\ar[ddl]\\
    M_2\ar@<0.5ex>[r]& X\ar@<.0ex>[l]\ar@<.0ex>[ur]\ar@<.0ex>[dr]&&\\
    && A\ar@<0.5ex>[ul]\ar[r]& N_1\ar[uul]
  }$$
  The only candidates for indecomposable rigid objects according to
  Lemma \ref{L1.1} are $M_1$ and $N_1$. We compute the support of
  $\underline{\Hom}(M_1, \ )$.

  $\xymatrix@C0.1cm@R0.3cm{
    &&&& N_1\ar[dr]&&N_1&\\
    &&& A\ar[dr]\ar[ur]&& B\ar[dr]\ar[ur]&& \cdots\\
    && X\ar[ur]\ar[rd]\ar[r]&M_2\ar[r]& X\ar[ur]\ar[rd]\ar[r]&M_2\ar[r]&X&\\
    & A\ar[ur]\ar[dr]&& B\ar[ur]\ar[dr]&& A\ar[ur]\ar[dr]&&\cdots\\
    M_1\ar[ur]&& N_1\ar[ur]&& M_1\ar[ur]&& N_1&
  }$ \hspace{0.5cm}
  $\xymatrix@C0.2cm@R0.3cm{
    &&&& 1\ar[dr]&&0&\\
    &&& 1\ar[dr]\ar[ur]&&1\ar[dr]\ar[ur]&&\cdots\\
    && 1\ar[ur]\ar[dr]\ar[r]&1\ar[r]& 1\ar[ur]\ar[r]\ar[dr]&0\ar[r]&1& \\
    & 1\ar[ur]\ar[dr]&& 0\ar[ur]\ar[dr]&& 1\ar[dr]\ar[ur]&&\cdots\\
    1\ar[ur]&& 0\ar[ur]&& 0\ar[ur]&&1&
  }$\\

  We see that $\underline{\Hom}(M_1,N_1)\neq 0$, so that
  $\Ext^1(M_1,M_1)\neq 0$ and $\Ext^1(N_1,N_1)\neq 0$.
\end{proof}
\begin{proposition}\label{p1.8}
  In the case $E_7$ there are two maximal rigid objects, which both
  are indecomposable, and neither of them is cluster tilting.
\end{proposition}
\begin{proof}
  We have the AR-quiver
  $$\xymatrix@C0.7cm@R0.9cm{
    &&& C\ar[dd]\ar@{--}[r]& D\ar[d]&&\\
    A\ar@{--}[d]\ar[r]& M_2\ar[r]\ar@{--}[d]\ar[dl]&
    Y_2\ar[dl]\ar@{--}[d]\ar[rr]&& Y_3\ar[ul]\ar[dll]\ar[r]&
    Y_1\ar@{--}[d]\ar[r]\ar[dll]& M_1\ar[dl]\ar@{--}[d]\\
    B\ar[r]& N_2\ar[ul]\ar[r]& X_2\ar[ul]\ar[r]& X_3\ar[uur]\ar@{--}[ur]\ar[ul]\ar[rr]&&
    X_1\ar[ul]\ar[r]& N_1\ar[ul]
  }$$
  Using Lemma \ref{L1.1}, we see that  the only candidates for
  indecomposable rigid objects are $A$, $B$, $M_1$, $N_1$, $C$ and
  $D$. We first compute the support of $\underline{\Hom}(A,\ )$.

$\xymatrix@C0.1cm@R0.3cm{
    &&&&& M_1\ar[dr]&&\\
    &&&&Y_1\ar[ur]\ar[dr]&& X_1&\cdots\\
    &&&Y_3\ar[ur]\ar[rd]\ar[r] &C\ar[r]& X_3\ar[ur]\ar[dr]\ar[r]&D&\\
    &&Y_2\ar[ur]\ar[dr]&& X_2\ar[ur]\ar[dr]&&Y_2&\\
    &M_2\ar[ur]\ar[dr]&& N_2\ar[ur]\ar[dr]&& M_2\ar[ur]\ar[dr]&&\cdots\\
    A\ar[ur]&& B\ar[ur]&& A\ar[ur]&&B&
  }$\hspace{0cm}
$\xymatrix@C0.1cm@R0.3cm{
    &&&&& 1\ar[dr]&& 0\ar[dr]&& 0\ar[dr]&& 1\ar[dr]&&
    0\ar[dr]&&&&&\\
    &&&& 1\ar[ur]\ar[dr]&& 1\ar[ur]\ar[dr]&& 0\ar[ur]\ar[dr]&&
    1\ar[ur]\ar[dr]&& 1\ar[ur]\ar[dr]&& 0\ar[dr]&&&&\\
    &&& 1\ar[ur]\ar[rd]\ar[r]&1\ar[r]& 1\ar[ur]\ar[rd]\ar[r]&0\ar[r]&
    1\ar[ur]\ar[dr]\ar[r]&1\ar[r]&
    1\ar[ur]\ar[rd]\ar[r]&0\ar[r]&1\ar[ur]\ar[dr]\ar[r]&1\ar[r]&
    1\ar[ur]\ar[rd]\ar[r]&0\ar[r]&0\ar[dr]&&&\\ 
    && 1\ar[ur]\ar[dr]&& 0\ar[ur]\ar[dr]&& 1\ar[ur]\ar[dr]&&
    1\ar[ur]\ar[dr]&& 1\ar[ur]\ar[dr]&& 0\ar[ur]\ar[dr]&&
    1\ar[ur]\ar[dr]&& 0\ar[dr] &&\\
    & 1\ar[ur]\ar[dr]&& 0\ar[ur]\ar[dr]&& 0\ar[ur]\ar[dr]&&
    1\ar[ur]\ar[dr]&& 1\ar[ur]\ar[dr]&& 0\ar[ur]\ar[dr]&&
    0\ar[ur]\ar[dr]&& 1\ar[ur]\ar[dr]&& 0\ar[dr]&\\    
    1\ar[ur]&& 0\ar[ur]&& 0\ar[ur]&& 0\ar[ur]&& 1\ar[ur]&& 0\ar[ur]&&
    0\ar[ur]&& 0\ar[ur]&& 1\ar[ur]&& 0
}$

We see that $\Ext^1(A,A)=0$, and so also $\Ext^1(B,B)=0$, so $A$ and
$B$ are rigid.

Next we compute the support of $\underline{\Hom}(M_1, \ )$.

$\xymatrix@C0.1cm@R0.3cm{
  M_1\ar[dr]&& N_1\ar[dr]&&M_1\ar[dr]&&N_1&\\
  & X_1\ar[ur]\ar[dr]&& Y_1\ar[dr]\ar[ur]&&X_1\ar[dr]\ar[ur]&&\cdots\\
  && Y_3\ar[ur]\ar[rd]\ar[r]&C\ar[r]& X_3\ar[dr]\ar[ur]\ar[r]&D\ar[r]&Y_3&\\
  &&& X_2\ar[ur]\ar[dr]&& Y_2\ar[dr]\ar[ur]&&\\
  &&&& M_2\ar[ur]\ar[dr]&& N_2&\cdots\\
  &&&&& B\ar[ur]&&
}$\hspace{0.1cm}
$\xymatrix@C0.2cm@R0.3cm{
  1\ar[dr]&& 0\ar[dr]&& 0\ar[dr]&& 1\\
  & 1\ar[ur]\ar[dr]&& 0\ar[dr]\ar[ur]&& 1\ar[ur]\ar[dr]&&\cdots\\
  && 1\ar[ur]\ar[r]\ar[dr]&1\ar[r]& 1\ar[dr]\ar[ur]\ar[r]&0\ar[r]&1\\
  &&& 1\ar[ur]\ar[dr]&&1\ar[dr]\ar[ur]&\\
  &&&& 1\ar[dr]\ar[ur]&&1&\cdots \\
  &&&&& 1\ar[ur]&
}$

We see that $\Ext^1(M_1,M_1)\neq 0$ and $\Ext^1(N_1,N_1)\neq 0$, so
that $M_1$ and $N_1$ are not rigid.

Then we compute the support of $\underline{\Hom}(C,\ )$.

$\xymatrix@C0.1cm@R0.3cm{
  &&& N_1\ar[dr]&&M_1\ar[dr]&&&\\
  && X_1\ar[ur]\ar[dr]&& Y_1\ar[dr]\ar[ur]&&X_1&\cdots\\
  C\ar[r]& X_3\ar[ur]\ar[r]\ar[dr]&D\ar[r]& Y_3\ar[ur]\ar[r]\ar[dr]&C\ar[r]&
  X_3\ar[dr]\ar[ur]\ar[r] &D&\\  
  && Y_2\ar[ur]\ar[dr]&& X_2\ar[ur]\ar[dr]&&Y_2&\\
  &&& N_2\ar[ur]\ar[dr]&& M_2\ar[ur]\ar[dr]&&\cdots\\
  &&&& A\ar[ur]&&B&
}$\hspace{0.5cm}
$\xymatrix@C0.2cm@R0.3cm{
  &&& 1\ar[dr]&& 0\ar[dr]&&\\
  && 1\ar[ur]\ar[dr]&& 1\ar[ur]\ar[dr]&& 1&\cdots\\
  1\ar[r]& 1\ar[ur]\ar[r]\ar[dr]&0\ar[r]& 1\ar[ur]\ar[r]\ar[dr]&1\ar[r]&
  2\ar[r]\ar[dr]\ar[ur]&1&\\  
  && 1\ar[ur]\ar[dr]&& 1\ar[dr]\ar[ur]&&2&\\
  &&& 1\ar[ur]\ar[dr]&&1\ar[ur]\ar[dr]&&\cdots\\
  &&&& 1\ar[ur]&&0&
}$

We see that $\Ext^1(C,C)\neq 0$ and $\Ext^1(D,D)\neq 0$, so that $C$
and $D$ are not rigid. Hence $A$ and $B$ are the rigid indecomposable
objects, and they are maximal rigid.

Since $\Ext^1(A,C)= 0$, we see that $A$ and hence $B$ is not cluster
tilting.
\end{proof}
\begin{proposition}\label{p1.9}
  In the case $E_8$ there are no indecomposable rigid objects.
\end{proposition}
\begin{proof}
  We have the AR-quiver
  $$\xymatrix@C0.5cm@R0.8cm{
    && A_2\ar[dd]\ar@{--}[r]& B_2\ar[d]&&&&\\
    N_2\ar[r]\ar@{--}[d]& D_2\ar[dl]\ar@{--}[d]\ar[rr]&&
    X_1\ar[ul]\ar@{--}[dl]\ar[r]\ar[dll]& X_2\ar[lld]\ar[dll]\ar@{--}[d]\ar[r]&
    C_1\ar[dl]\ar@{--}[d]\ar[r]& B_1\ar[dl]\ar@{--}[d]\ar[r]&
    N_1\ar[dl]\ar@{--}[d]\\
    M_2\ar[r]& C_2\ar[ul]\ar[r]& Y_1\ar[ul]\ar[uur]\ar[rr]&&
    Y_2\ar[ul]\ar[r]& D_1\ar[ul]\ar[r]& A_1\ar[ul]\ar[r]& M_1\ar[ul]
  }$$
  The only candidates for indecomposable rigid objects are $M_1$, $N_1$, $M_2$, $N_2$, $A_2$ and $B_2$, by Lemma \ref{L1.1}. We first compute the support of
  $\underline{\Hom}(M_1,\ )$:
  
  $\xymatrix@C0.1cm@R0.3cm{
    &&&&&& M_2\ar[dr]&&N_2\\
    &&&&& D_2\ar[ur]\ar[dr] &&C_2\ar[dr]\ar[ur]&&\cdots\\
    &&&& Y_1\ar[ur]\ar[r]\ar[rd]&B_2\ar[r]& X_1
    \ar[ur]\ar[dr]\ar[r]&A_2\ar[r]& Y_1\\
    &&& X_2\ar[ur]\ar[dr]&& Y_2\ar[ur]\ar[dr]&& X_2\ar[ur]\ar[dr]&\\
    && D_1\ar[ur]\ar[dr]&& C_1\ar[ur]\ar[dr]&& D_1\ar[ur]\ar[dr]&&C_1\\
    & B_1\ar[ur]\ar[dr]&& A_1\ar[ur]\ar[dr] && B_1\ar[ur]\ar[dr]&&A_1\ar[ur]\ar[dr]&&\cdots\\
    M_1\ar[ur]&& N_1\ar[ur]&& M_1\ar[ur]&&N_1\ar[ur]&& M_1&&
  }$ \hspace{0.1cm}
$\xymatrix@C0.2cm@R0.3cm{
    &&&&&& 1\ar[dr]&&0\ar[dr]&&\\
    &&&&& 1\ar[ur]\ar[dr] && 1\ar[dr]\ar[ur]&&0&\cdots\\
    &&&&1\ar[ur]\ar[dr]\ar[r]&1\ar[r]&1\ar[ur]\ar[dr]\ar[r]&0\ar[r]
    &1\ar[dr]\ar[ur]\ar[r]&1&\\ 
    &&& 1\ar[ur]\ar[dr]&& 0\ar[ur]\ar[dr]&& 1\ar[dr]\ar[ur]&&1&\\
    && 1\ar[ur]\ar[dr]&& 0\ar[ur]\ar[dr]&& 0\ar[ur]\ar[dr]&& 1\ar[ur]\ar[dr]&&\\
    & 1\ar[ur]\ar[dr]&& 0\ar[ur]\ar[dr] && 0\ar[ur]\ar[dr]&&
    0\ar[ur]\ar[dr]&& 1\ar[dr]&\cdots\\
    1\ar[ur]&& 0\ar[ur]&& 0\ar[ur]&& 0\ar[ur]&& 0\ar[ur] &&1
  }$

We see that $\Ext^1(M_1,M_1)\neq 0$, and hence $\Ext^1(N_1,N_1)\neq
0$. 

Next we compute the support of $\underline{\Hom}(M_2, \ )$:\\
$\xymatrix@C0.1cm@R0.3cm{
    M_2\ar[dr]&& N_2\ar[dr]&& M_2\ar[dr]&& N_2&&\\
    &C_2\ar[dr]\ar[ur]&& D_2\ar[dr]\ar[ur]&& C_2\ar[ur]\ar[dr]&&\cdots&\\ 
    && Y_1\ar[dr]\ar[ur]\ar[r]&B_2\ar[r]& X_1\ar[ur]\ar[dr]\ar[r]&A_2\ar[r]&Y_1&&  \\
    &&&Y_2\ar[dr]\ar[ur]&& X_2\ar[dr]\ar[ur]&&&\\
    &&&& D_1 \ar[dr]\ar[ur]&& C_1&&\\
    &&&&& A_1\ar[dr]\ar[ur]&&\cdots& \\
    &&&&&& M_1 &&
  }$\hspace{0.5cm}
$\xymatrix@C0.2cm@R0.3cm{
    1\ar[dr]&& 0\ar[dr]&& 0\ar[dr]&& 1\\
    & 1\ar[ur]\ar[dr]&& 0\ar[ur]\ar[dr]&& 1\ar[ur]\ar[dr]&&\cdots\\
    && 1\ar[ur]\ar[r]\ar[dr]&1\ar[r]& 1\ar[ur]\ar[r]\ar[dr]&0\ar[r]&1\\
    &&& 1\ar[ur]\ar[dr]&&1\ar[dr]\ar[ur]&\\
    &&&& 1\ar[dr]\ar[ur]&&1\\
    &&&&& 1\ar[dr]\ar[ur]&&\cdots\\
    &&&&&& 1
  }$

We see that $\Ext^1(M_2,M_2)\neq 0$, and hence $\Ext^1(N_2,N_2)\neq 0$.

Finally we compute the support of $\underline{\Hom}(A_2,\ )$:\\
$
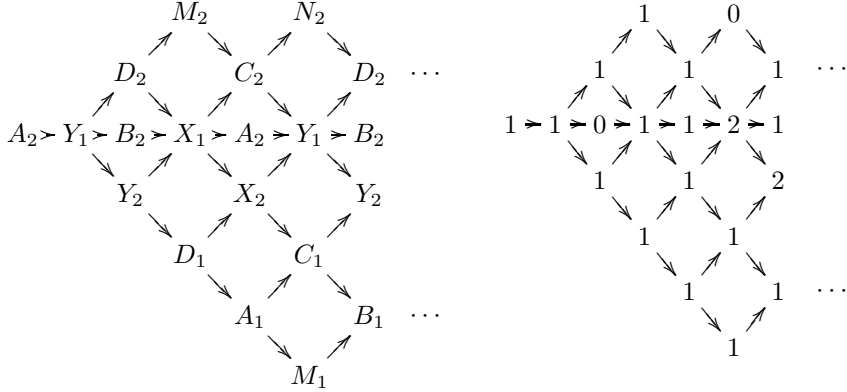
\xymatrix@C0.1cm@R0.3cm{
  &&&M_2\ar[dr]&& N_2\ar[dr]&\\
  && D_2\ar[ur]\ar[dr]&& C_2\ar[ur]\ar[dr]&&D_2&\cdots\\
  A_2\ar[r]& Y_1\ar[ur]\ar[dr]\ar[r]&B_2\ar[r]&
  X_1\ar[ur]\ar[dr]\ar[r]&A_2\ar[r]&Y_1\ar[r]\ar[ur]\ar[dr]&B_2\\
  && Y_2\ar[ur]\ar[dr]&& X_2\ar[ur]\ar[dr]&&Y_2\\
  &&& D_1\ar[ur]\ar[dr]&&C_1\ar[ur]\ar[dr]&\\
  &&&& A_1\ar[dr]\ar[ur]&&B_1&\cdots\\
  &&&&& M_1\ar[ur]&
}$\hspace{0.5cm}
$\xymatrix@C0.2cm@R0.3cm{
  &&&1\ar[dr]&& 0\ar[dr]&\\
  && 1\ar[ur]\ar[dr]&& 1\ar[ur]\ar[dr]&&1&\cdots\\
  1\ar[r]&1\ar[ur]\ar[dr]\ar[r]&0\ar[r]&1\ar[ur]\ar[dr]\ar[r]&1\ar[r]&
  2\ar[ur]\ar[r]\ar[dr]&1&\\
  && 1\ar[ur]\ar[dr]&& 1\ar[ur]\ar[dr]&&2\\
  &&& 1\ar[ur]\ar[dr]&&1\ar[dr]\ar[ur]&\\
  &&&& 1\ar[dr]\ar[ur]&&1&\cdots\\
  &&&&& 1\ar[ur] &
}$

It follows that $\Ext^1(A_2,A_2)\ne 0$, and similarly $\Ext^1(B_2,B_2)\ne 0$. Hence there are no indecomposable rigid objects.
\end{proof}

\section{Computation with \tt{Singular}} \label{sec2}

An alternative way to carry out computations of $\Ext^1$--spaces in the stable category of 
maximal Cohen-Macaulay modules is to use the computer algebra system {\tt Singular}, see
\cite{GP}.  
Let $$R = k[x_1,x_2,\dots,x_n]_{\langle x_1, x_2,\dots, x_n \rangle}/I$$
 be a Cohen-Macaulay local ring which is an isolated singularity, and $M$ and $N$ two maximal Cohen-Macaulay modules. Denote by $\widehat{R}$ the completion of $R$. Since all the spaces $\Ext^i_R(M,N)$ ($ i \ge 1$) are finite-dimensional over $k$ and the functor $\mod(R)\to \mod(\widehat{R})$ is exact, maps the maximal Cohen-Macaulay modules to  maximal Cohen-Macaulay modules and the finite length modules to finite length modules, we can conclude that 
$$
\dim_k(\Ext^i_R(M,N)) = \dim_k(\Ext^i_{\widehat{R}}(\widehat{M},\widehat{N})).
$$
As an illustration we show how to do this  for the case $E_7$.

\begin{proposition}\label{p2.1}
In the case $E_7$ there are two maximal rigid  objects, which both are indecomposable 
and neither of them is cluster tilting. 
\end{proposition}

\noindent
By \cite{Y} the AR-quiver of $\underline{\CM}(R)$ has the form
$$
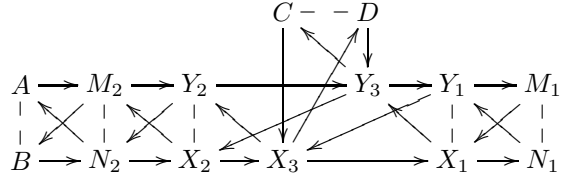
\xymatrix@C0.5cm@R0.5cm{
    &&& C\ar[dd]\ar@{--}[r]& D\ar[d]&&\\
    A\ar@{--}[d]\ar[r]& M_2\ar[r]\ar@{--}[d]\ar[dl]&
    Y_2\ar[dl]\ar@{--}[d]\ar[rr]&& Y_3\ar[ul]\ar[dll]\ar[r]&
    Y_1\ar@{--}[d]\ar[r]\ar[dll]& M_1\ar[dl]\ar@{--}[d]\\
    B\ar[r]& N_2\ar[ul]\ar[r]& X_2\ar[ul]\ar[r]& X_3\ar[ul]\ar[rr]\ar[uur]&&
    X_1\ar[ul]\ar[r]& N_1\ar[ul]
}
$$
By Lemma \ref{L1.1} only the modules $A, B, C, D, M_1, N_1$ can be rigid. 
Since $B = \tau(A), D = \tau(C)$, $N_1 = \tau(M_1)$, the pairs of modules 
$(A,B)$, $(C,D)$ and $(M_1, N_1)$ are rigid or not rigid simultaneously. By
\cite{Y} we have the following presentations:

$$
R \xrightarrow{x^2 +y^3} R 
\xrightarrow{x} R \lar A \lar 0,
$$
$$
R^2 \xrightarrow{{x \ \ y\choose y^2\ -x}}
R^2 \xrightarrow{x{x\ \ y\choose y^2\ -x}}
R^2 \lar C \lar 0,
$$
$$
R^2 \xrightarrow{{x^2\ \ \ y\choose xy^2\ -x}}
R^2 \xrightarrow{{x\ \ \ y\choose x y^2\ -x^2}}
R^2 \lar M_1 \lar 0,
$$
so we can use the computer algebra system {\tt Singular} in order to 
compute the $\Ext^1$--spaces between these modules.

\noindent
{\tt > Singular}  (call  the program {\tt ``Singular''})\\
{\tt > LIB ``homolog.lib'';} (call the library of homological algebra)\\
{\tt > ring S = 0,(x,y),ds;}   (defines the ring $S = \mathbb{Q}[x,y]_{\langle x, y\rangle}$)\\
{\tt > ideal I = x3 + xy3;}     (defines the ideal $x^3 + xy^3$ in $S$) \\
{\tt > qring R = std(I);} (defines the ring $\mathbb{Q}[x,y]_{\langle x, y\rangle}/I)$\\
{\tt > module A = [x];} \\
{\tt > module C = [x2, xy2], [xy, -x2];} \\
{\tt > module M1 = [x2, xy2], [y, -x2];}  (define modules $A, C, M_1$) \\
{\tt > list l = Ext(1,A,A,1);} \\
{\tt // dimension of $\Ext^1$:  -1} (Output: $\Ext^1_R(A,A) = 0$)\\
{\tt > list l = Ext(1,C,C,1);} \\
{\tt // ** redefining l **}    \\
{\tt // dimension of $\Ext^1$: 0} (the Krull dimension of $\Ext^1_R(C,C)$ is 0) \\
{\tt // vdim of $\Ext^1$: 2}\ \ \ ($\dim_k(\Ext^1_R(C,C)) = 2$)  \\
{\tt > list l = Ext(1,M1,M1,1);} \\
{\tt // ** redefining l **} \\
{\tt // dimension of $\Ext^1$: 0}\\
{\tt // vdim of $\Ext^1$: 10}\\
{\tt > list l = Ext(1,A,C,1);} \\
{\tt // ** redefining l **} \\
{\tt // dimension of $\Ext^1$: -1} \\

This computation shows that the modules $A$ and $B$ are rigid, $C,D, M_1$ and $N_1$ are not 
rigid and since $\Ext^1_R(A, C) = 0$, there are no cluster tilting objects 
in the 
stable category $\underline{\CM}(R)$.  

\section{One-dimensional hypersurface singularities}\label{onedim}

We shall construct a large class of one-dimensional
hypersurface singularities having a cluster tilting object, then
classify all cluster tilting objects.
Our method is based on the higher theory of almost split sequences and
Auslander algebras studied in \cite{I1,I2}. We also use a relationship
between cluster tilting objects in $\CM(R)$ and tilting modules over
the endomorphism algebra of a cluster tilting object \cite{I2}. Then
we shall compare cluster tilting mutation given in Definition \ref{cluster tilting mutation}
with tilting mutation by using results due to Riedtmann-Schofield \cite{RS} and Happel-Unger \cite{HU1,HU2}.

In this section, we usually consider cluster tilting objects in $\CM(R)$
instead of $\ul{\CM}(R)$.

Let $k$ be an infinite field, $S:=k[[x,y]]$ and ${\mathfrak m}:=(x,y)$.
We fix $f\in{\mathfrak m}$ and write $f=f_1\cdots f_n$ for irreducible
formal power series $f_i\in{\mathfrak m}$ ($1\le i\le n$).
Put
$$S_i:=S/(f_1\cdots f_i)\ \mbox{ and }\ R:=S_n=S/(f).$$
We assume that $R$ is reduced, so we have $(f_i)\neq(f_j)$ for any $i\neq j$, and $R$ is then an isolated singularity.

Our main results in this section are the following, where the part (a) remains true in any dimension.

\begin{theorem}\label{XAA}
\begin{itemize}
\item[(a)] $\bigoplus_{i=1}^nS_i$ is a rigid object in $\CM(R)$.
\item[(b)] $\bigoplus_{i=1}^nS_i$ is a cluster tilting object in $\CM(R)$ if the following condition (A) is satisfied.

(A) $f_i\notin{\mathfrak m}^2$ for any $1\le i\le n$.
\end{itemize}
\end{theorem}

Let ${\mathfrak S}_n$ be the symmetric group of degree $n$. For $w\in{\mathfrak S}_n$ and 
$I\subseteq\{1,\cdots,n\}$, we put
\[S^w_i:=S/(f_{w(1)}\cdots f_{w(i)}),\ M_w:=\bigoplus_{i=1}^nS^w_i\
\mbox{ and }\ S_I:=S/(\prod_{i\in I}f_i).\]

\begin{theorem}\label{XAB}
Assume that (A) is satisfied.
\begin{itemize}
\item[(a)] There are exactly $n!$ basic cluster tilting objects $M_w$
($w\in{\mathfrak S}_n$) and exactly $2^n-1$ indecomposable rigid objects
$S_I$ ($\emptyset\neq I\subseteq\{1,\cdots,n\}$) in $\CM(R)$.
\item[(b)] For any $w\in{\mathfrak S}_n$, there are exactly $n!$ basic Cohen-Macaulay
tilting $\End_R(M_w)$-modules

\noindent
$\Hom_R(M_w,M_{w'})$ ($w'\in{\mathfrak S}_n$)
of projective dimension at most one.
Moreover, all algebras

\noindent
$\End_R(M_w)$ ($w\in{\mathfrak S}_n$) are derived
equivalent.
\end{itemize}
\end{theorem}

It is interesting to compare with results in \cite{IR}, where
two-dimensional (2-Calabi-Yau) algebras $\Gamma$ are treated and a
bijection between elements in an affine Weyl group and tilting
$\Gamma$-modules of projective dimension at most one is given.
Here the algebra is one-dimensional, and Weyl groups appear.

Here we consider three examples.
\begin{itemize}
\item[(a)] Let $R$ be a curve singularity of type $A_{2n-1}$ or $D_{2n+2}$, so
\[R=S/((x-y^n)(x+y^n))\ \mbox{ or }\ R=S/(y(x-y^n)(x+y^n)).\]
By our theorems, there are exactly $2$ or $6$ cluster tilting objects 
and exactly $3$ or $7$ indecomposable rigid objects in $\CM(R)$, which
fits with our computations in section 1.

\item[(b)] Let $R$ be a curve singularity of type $T_{3,2q+2}(\lambda)$ or
$T_{2p+2,2q+2}(\lambda)$, so
\begin{eqnarray*}
R=S/((x-y^2)(x-y^q)(x+y^q))&&(R=S/(y(y-x^2)(y-\lambda x^2))\mbox{
  for }q=2),\\
R=S/((x^p-y)(x^p+y)(x-y^q)(x+y^q))&&(R=S/(xy(x-y)(x-\lambda
y))\mbox{ for }p=q=1).
\end{eqnarray*}
By our theorems, there are exactly $6$ or $24$ cluster tilting objects 
and exactly $7$ or $15$ indecomposable rigid objects in $\CM(R)$.

\item[(c)] Let $\lambda_i\in k$ ($1\le i\le n$) be mutually distinct elements
in $k$. Put
\[R:=S/((x-\lambda_1y)\cdots(x-\lambda_ny)).\]
By our theorems, there are exactly $n!$ cluster tilting objects and
exactly $2^n-1$ indecomposable rigid objects in $\CM(R)$.
\end{itemize}

\medskip
First of all, Theorem \ref{XAA}(a) follows immediately from the
following observation.

\begin{proposition}\label{XAD}
For $g_1,g_2\in{\mathfrak m}$ and $g_3\in S$, put $R:=S/(g_1g_2g_3)$.
If $g_1$ and $g_2$ have no common factor, then
$\Ext^1_R(S/(g_1g_3),S/(g_1))=0=\Ext^1_R(S/(g_1),S/(g_1g_3))$.
\end{proposition}

\begin{proof}
We have a projective resolution
\[R\stackrel{g_2}{\to}R\stackrel{g_1g_3}{\to}R\to S/(g_1g_3)\to 0.\]
Applying $\Hom_R(\ ,S/(g_1))$, we have a complex
\[S/(g_1)\stackrel{g_1g_3=0}{\longrightarrow}S/(g_1)\stackrel{g_2}{\to}S/(g_1).\]
This is exact since $g_1$ and $g_2$ have no common factor.
Thus we have the former equation, and the other one can be proved
similarly.
\end{proof}

Our plan of proof of Theorem \ref{XAA}(b) is the following.
\begin{itemize}
\item[(i)] First we shall prove Theorem \ref{XAA} under the following
  stronger assumption:

(B) ${\mathfrak m}=(f_1,f_2)=\cdots=(f_{n-1},f_n)$.

\item[(ii)] Then we shall prove the general statement of Theorem \ref{XAA}.
\end{itemize}

We need the following general result in \cite{I1,I2}.

\begin{proposition}\label{XAE}
Let $R$ be a complete local Gorenstein ring of dimension at most three and
$M$ a rigid Cohen-Macaulay $R$-module which is a generator (i.e. $M$ contains $R$ as a direct summand).
Then the following conditions are equivalent.
\begin{itemize}
\item[(a)] $M$ is a cluster tilting object in $\CM(R)$.
\item[(b)] $\gl\End_R(M)\le 3$.
\item[(c)] For any $X\in\CM(R)$, there exists an exact sequence $0\to M_1\to
M_0\to X\to0$ with $M_i\in\add M$.
\item[(d)] For any indecomposable direct summand $X$ of $M$, there
  exists an exact sequence $0\to
  M_2\stackrel{c}{\to}M_1\stackrel{b}{\to}M_0\stackrel{a}{\to}X$
  with $M_i\in\add M$ and $a$ is a right almost split map in $\add M$.
\end{itemize}
\end{proposition}

\begin{proof}
(a)$\Leftrightarrow$(b) For $d=\dim R$, take the $d$-cotilting module $T=R$ and apply \cite[Th. 5.1(3)]{I2} for $m=d$ and $n=2$ there.

(a)$\Leftrightarrow$(c) See \cite[Prop. 2.2.2]{I1}.

(a)$\Rightarrow$(d) See \cite[Th. 3.3.1]{I1}.

(d)$\Rightarrow$(b) For any simple $\End_R(M)$-module $S$, there
exists an indecomposable direct summand $X$ of $M$ such that $S$ is
the top of the projective $\Hom_R(M,X)$. Since $\Ext^1_R(M,M_2)=0$,
the sequence in (d) gives a projective resolution
$0\to\Hom_R(M,M_2)\to\Hom_R(M,M_1)\to\Hom_R(M,M_0)\to\Hom_R(M,X)\to
S\to 0$. Thus we have $\pd S\le 3$ and $\gl\End_R(M)\le 3$.
\end{proof}

The sequence in (d) is called a {\it 2-almost split sequence} when $X$
is non-projective and $a$ and $b$ are right minimal.
In this case $a$ is surjective, $c$ is a left almost split map in $\add M$, and $b$ and $c$ are left minimal.
There is a close relationship between 2-almost
split sequences and exchange sequences \cite{IY}.

We shall construct exact sequences satisfying the above condition (d) 
in Lemma \ref{XAF} and Lemma \ref{XAG} below. 

We use the isomorphism
\[\Hom_R(S_i,S_j)\simeq\left\{\begin{array}{cc}
(f_{i+1}\cdots f_j)/(f_1\cdots f_j)&i<j\\
S/(f_1\cdots f_j)&i\ge j.\end{array}\right.\]

\begin{lemma}\label{XAF}
Let $R=S/(f)$ be a one-dimensional reduced hypersurface singularity, $S_0:=0$ and $1\le i<n$.
\begin{itemize}
\item[(a)] We have exchange sequences (see Definition \ref{cluster tilting mutation})
\begin{eqnarray*}
&0\to S_i\xrightarrow{{f_{i+1}\choose -1}}S_{i+1}\oplus S_{i-1}
\xrightarrow{(1\ f_{i+1})}
S/(f_1\cdots f_{i-1}f_{i+1})\to 0,&\\
&0\to S/(f_1\cdots f_{i-1}f_{i+1})\xrightarrow{{f_i\choose 1}}S_{i+1}\oplus
S_{i-1}\xrightarrow{(-1\ f_i)}S_i\to 0.&
\end{eqnarray*}
\item[(b)] If $(f_i,f_{i+1})=\mathfrak{m}$, then we have a 2-almost split sequence
\[0\to S_i\xrightarrow{{f_{i+1}\choose -1}}S_{i+1}\oplus S_{i-1}
\xrightarrow{{f_i\ f_if_{i+1}\choose 1\ \ \ f_{i+1}}}S_{i+1}\oplus
S_{i-1}\xrightarrow{(-1\ f_i)}S_i\to 0\]
in $\add\bigoplus_{i=1}^nS_i$.
\end{itemize}
\end{lemma}

\begin{proof}
(a) Consider the map $a:=(-1\ f_i):S_{i+1}\oplus S_{i-1}\to S_i$.
Any morphism from $S_j$ to $S_i$ factors through $1:S_{i+1}\to
S_i$ (respectively, $f_i:S_{i-1}\to S_i$) if $j>i$ (respectively, $j<i$).
Thus $a$ is a minimal right $(\add\bigoplus_{j\neq i}S_j)$-approximation.

It is easily checked that $\Ker a=\{s\in S_{i+1}\ |\ \overline{s}\in
f_iS_i\}=(f_i)/(f_1\cdots f_{i+1})\simeq S/(f_1\cdots
f_{i-1}f_{i+1})$, where we denote by $\overline{s}$ the image of 
$s$ via the natural surjection $S_{i+1}\to S_i$. 

Consider the surjective map $b:=(1\ f_{i+1}):S_{i+1}\oplus
S_{i-1}\to S/(f_1\cdots f_{i-1}f_{i+1})$.
It is easily checked that $\Ker b=\{s\in S_{i+1}\ |\ \overline{s}\in
(f_{i+1})/(f_1\cdots f_{i-1}f_{i+1})\}=(f_{i+1})/(f_1\cdots
f_{i+1})\simeq S_i$, where we denote by $\overline{s}$ the image of
$s$ via the natural surjection $S_{i+1}\to S/(f_1\cdots f_{i-1}f_{i+1})$. 

(b) This sequence is exact by (a).
Any non-isomorphic endomorphism of $S_i$ is multiplication with an
element in ${\mathfrak m}$, which is equal to $(f_i,f_{i+1})$ by our assumption.
Since $f_{i+1}$ (respectively, $f_i$) $:S_i\to S_i$ factors through
$1:S_{i+1}\to S_i$ (respectively, $f_i:S_{i-1}\to S_i$), we have that $a$ is a
right almost split map.
\end{proof}

Now we choose $f_{n+1}\in{\mathfrak m}$
such that ${\mathfrak m}=(f_n,f_{n+1})$, and $f_{n+1}$ and $f_1\cdots
f_n$ have no common factor. This is possible by our assumption (A).

\begin{lemma}\label{XAG}
We have an exact sequence
\[0\to S_{n-1}\xrightarrow{{f_n\choose -f_{n+1}}}S_n\oplus S_{n-1}
\xrightarrow{(f_{n+1}\ f_n)}S_n\]
with a minimal right almost split map $(f_{n+1}\ f_n)$ in $\add\bigoplus_{i=1}^nS_i$.
\end{lemma}

\begin{proof}
Consider the map $a:=(f_{n+1}\ f_n):S_n\oplus S_{n-1}\to S_n$.
Any morphism from $S_j$ ($j<n$) to $S_n$ factors through $f_n:S_{n-1}\to S_n$.

Any non-isomorphic endomorphism of $S_n$ is multiplication with an
element in ${\mathfrak m}=(f_{n+1},f_n)$.
Since $f_n:S_n\to S_n$ factors through $f_n:S_{n-1}\to S_n$, we
have that $a$ is a right almost split map. 

It is easily checked that $\Ker a=\{s\in S_{n-1}\ |\ f_ns\in
f_{n+1}S_n\}=(f_{n+1},f_1\cdots f_{n-1})/(f_1\cdots f_{n-1})$, which
is isomorphic to $S_{n-1}$ by the choice of $f_{n+1}$.
In particular, $a$ is right minimal.
\end{proof}

Thus we finished the proof of Theorem \ref{XAA} under the stronger assumption
(B).

To show the general statement of Theorem \ref{XAA}, we need some
preliminary observations.
Let us consider cluster tilting mutation in $\CM(R)$.
We use the notation introduced at the beginning of this section.

\begin{lemma}\label{XAH}
For $w\in{\mathfrak S}_n$, we assume that $M_w$ is a cluster tilting object
in $\CM(R)$. Then, for $1\le i<n$ and $s_i=(i\ i+1)$, we have exchange sequences
\begin{eqnarray*}
0\to S^w_i\to S^w_{i+1}\oplus S^w_{i-1}\to S^{ws_i}_i\to0
\ \mbox{ and } \ 0\to S^{ws_i}_i\to S^w_{i+1}\oplus S^w_{i-1}\to S^w_i\to0.
\end{eqnarray*}
\end{lemma}

\begin{proof}
Without loss of generality, we can assume $w=1$.
Then the assertion follows from Lemma \ref{XAF}(a).
\end{proof}

Immediately, we have the following.

\begin{proposition}\label{XAI}
Assume that $M_w$ is a cluster tilting object in $\CM(R)$ for some
$w\in{\mathfrak S}_n$.
\begin{itemize}
\item[(a)] The cluster tilting mutations of $M_w$ are $M_{ws_i}$ ($1\le i<n$).
\item[(b)] $M_{w'}$ is a cluster tilting object in $\CM(R)$ for any
  $w'\in{\mathfrak S}_n$.
\end{itemize}
\end{proposition}

\begin{proof}
(a) This follows from Lemma \ref{XAH}.

(b) This follows from (a) since ${\mathfrak S}_n$ is generated by
$s_i$ ($1\le i<n$).
\end{proof}

The following result is also useful.

\begin{lemma}\label{XAN}
Let $R$ and $R'$ be complete local Gorenstein rings with $\dim R=\dim R'$
and $M$ a rigid object in $\CM(R)$ which is a generator.
Assume that there exists a surjection $R'\to R$, and
we regard $\CM(R)$ as a full subcategory of $\CM(R')$. 
If $R'\oplus M$ is a cluster tilting object in $\CM(R')$, then
$M$ is a cluster tilting object in $\CM(R)$.
\end{lemma}

\begin{proof}
We use the equivalence (a)$\Leftrightarrow$(c) in Proposition \ref{XAE}, which remains true in any dimension \cite[Prop. 2.2.2]{I1}.
For any $X\in\CM(R)$, take a right $(\add M)$-approximation $f:M_0\to X$ of $X$.
Since $M$ is a generator of $R$, we have an exact sequence
$0\to Y\to M_0\stackrel{f}{\to}X\to 0$ with $Y\in\CM(R)$.
Since $R'$ is a projective $R'$-module,
$f$ is a right $\add(R'\oplus M)$-approximation of $X$.
Since $R'\oplus M$ is a cluster tilting object in $\CM(R')$,
we have $Y\in\add(R'\oplus M)$. Since $Y\in\CM(R)$, we have $Y\in\add M$.
Thus $M$ satisfies condition (c) in Proposition \ref{XAE}.\end{proof}

Now we shall prove Theorem \ref{XAA}.
Since $k$ is an infinite field and the assumption (A) is satisfied, we can take irreducible formal power series
$g_i\in{\mathfrak m}$ ($1\le i<n$) such that $h_{2i-1}:=f_i$ and $h_{2i}:=g_i$
satisfy the following conditions:
\begin{itemize}
\item $(h_i)\neq(h_j)$ for any $i\neq j$.
\item ${\mathfrak m}=(h_1,h_2)=(h_2,h_3)=\cdots=(h_{2n-2},h_{2n-1})$.
\end{itemize}

Put $R':=S/(h_1\cdots h_{2n-1})$. This is reduced by the first condition.

Since we have already proved Theorem \ref{XAA} under the assumption (B),
we have that \newline$\bigoplus_{i=1}^{2n-1}S/(h_1\cdots h_i)$
is a cluster tilting object in $\CM(R')$.
By Proposition \ref{XAI}, \newline$\bigoplus_{i=1}^{2n-1}S/(h_{w(1)}\cdots h_{w(i)})$
is a cluster tilting object in $\CM(R')$ for any $w\in{\mathfrak S}_{2n-1}$.
In particular,
\[(\bigoplus_{i=1}^nS/(f_1\cdots f_i))\oplus
(\bigoplus_{i=1}^{n-1}S/(f_1\cdots f_ng_1\cdots g_i))\]
is a cluster tilting object in $\CM(R')$. Moreover we have surjections
\[R'\to\cdots\to S/(f_1\cdots f_ng_1g_2)\to S/(f_1\cdots f_ng_1)\to R.\]
Using Lemma \ref{XAN} repeatedly, we have that
$\bigoplus_{i=1}^n(S/(f_1\cdots f_i))$ is a cluster tilting object in $\CM(R)$.
Thus we have proved Theorem \ref{XAA}.\qed

\bigskip
Before proving Theorem \ref{XAB}, we give the following description of the quiver of the endomorphism algebras.

\begin{proposition}
Assume that (A) is satisfied.
\begin{itemize}
\item[(a)] The quiver of $\End_R(\bigoplus_{i=1}^nS_i)$ is
\[\xymatrix@C0.5cm@R0.5cm{
    &S_1\ar@<0.5ex>[r]& S_2\ar@<0.5ex>[l]\ar@<0.5ex>[r]& \cdots\ar@<0.5ex>[l]\ar@<0.5ex>[r]& S_{n-1}\ar@<0.5ex>[l]\ar@<0.5ex>[r]&S_n\ar@<0.5ex>[l]\ar@(dr,ur)[]&}\]
where in addition there is a loop at $S_i$ ($1\le i<n$) if and only if $(f_i,f_{i+1})\neq\mathfrak{m}$.
\item[(b)] We have the quiver of $\underline{\End}_R(\bigoplus_{i=1}^nS_i)$ by removing the vertex $S_n$ from the quiver in (a).
\end{itemize}
\end{proposition}

\begin{proof}
We only have to show (a). We only have to calculate minimal right almost split maps in $\add\bigoplus_{i=1}^nS_i$.
We have a minimal right almost split map $S_n\oplus S_{n-1}\to S_n$ by Lemma \ref{XAG}.
If $(f_i,f_{i+1})=\mathfrak{m}$ ($1\le i<n$), then we have a minimal right
almost split map $S_{i+1}\oplus S_{i-1}\to S_i$ by Lemma \ref{XAF}.

We only have to consider the case $(f_i,f_{i+1})\neq\mathfrak{m}$ ($1\le i<n$).
Take $g\in\mathfrak{m}$ such that $(f_i,f_{i+1},g)=\mathfrak{m}$.
It is easily check (cf proof of Lemma \ref{XAF}) that we have a right almost split map
\[c:=(-1\ g\ f_i):S_{i+1}\oplus S_i\oplus S_{i-1}\to S_i.\]

Assume that $c$ is not right minimal. Then there exists a right almost split map
of the form $c':S_{i+1}\oplus S_i\to S_i$ ($i>1$), $S_{i}\oplus S_{i-1}\to S_i$ or $S_{i+1}\oplus S_{i-1}\to S_i$.
For the first case, it is easily checked that $f_i:S_{i-1}\to S_i$ does not factor through $c'$, a contradiction.
Similarly we have the contradiction for the remaining cases.
Thus $c$ is the minimal right almost split map.
\end{proof}

\bigskip
In the rest we shall show Theorem \ref{XAB}.
We recall results on tilting mutation due to Riedtmann-Schofield
\cite{RS} and Happel-Unger \cite{HU1,HU2}.
For simplicity, a tilting module means a tilting module of projective
dimension at most one.

Let $\Gamma$ be a module-finite algebra with $n$ simple modules over a complete local ring
with $n$ simple modules. Their results remain valid in this setting.
Recall that, for basic tilting $\Gamma$-modules $T$ and $U$,
we write
\[T\ge U\]
if $\Ext^1_\Gamma(T,U)=0$. 
By tilting theory, we have $\Fac T=\{X\in\mod(\Gamma)\ |\ \Ext^1_\Gamma(T,X)=0\}$.
Thus $\Ext^1_\Gamma(T,U)=0$ is equivalent to $\Fac T\supset\Fac U$,
and $\ge$ gives a partial order.
On the other hand, we call a $\Gamma$-module $T$ {\it almost complete tilting}
if $\pd{}_\Gamma T\le 1$, $\Ext^1_\Gamma(T,T)=0$ and $T$ has exactly $(n-1)$
non-isomorphic indecomposable direct summands.

We collect some basic results.

\begin{proposition}\label{XAJ}
\begin{itemize}
\item[(a)] Any almost complete tilting $\Gamma$-module has at most two complements.
\item[(b)] $T$ and $U$ are neighbors in the partial order if and only if
there exists an almost complete tilting $\Gamma$-module which is
a common direct summand of $T$ and $U$.
\item[(c)] Assume $T\ge U$. Then there exists a sequence $T=T_0>T_1>T_2>\cdots>U$
satisfying the following conditions.
\begin{itemize}
\item[(i)] $T_i$ and $T_{i+1}$ are neighbors.

\item[(ii)] Either $T_i=U$ for some $i$ or the sequence is infinite.
\end{itemize}
\item[(d)] $T\ge U$ if and only if there exists an exact
sequence $0\to T\to U_0\to U_1\to 0$ with $U_i\in\add U$.
\end{itemize}
\end{proposition}

If the conditions in (b) above are satisfied, we call $T$ a {\it tilting mutation} of $U$.

\medskip
We also need the following easy observation on Cohen-Macaulay
tilting modules.
For a module-finite $R$-algebra $\Gamma$, we call a $\Gamma$-module {\it Cohen-Macaulay} if it is a Cohen-Macaulay $R$-module. As usual, we denote by $\CM(\Gamma)$ the category of Cohen-Macaulay $\Gamma$-modules.

\begin{lemma}\label{XAK}
Let $\Gamma$ be a module-finite algebra over a complete local Gorenstein
ring $R$ such that $\Gamma\in\CM(R)$, and $T$ and $U$ tilting $\Gamma$-modules. Assume $U\in\CM(\Gamma)$.
\begin{itemize}
\item[(a)] If $T\ge U$, then $T\in\CM(\Gamma)$.
\item[(b)] Let $P$ be a projective $\Gamma$-module such that $\Hom_R(P,R)$ is a
projective $\Gamma^{\op}$-module. Then $P\in\add U$.
\end{itemize}
\end{lemma}

\begin{proof}
(a) By Proposition \ref{XAJ}(d), there exists an exact
sequence $0\to T\to U_0\to U_1\to 0$ with $U_i\in\add U$. Thus the assertion holds.

(b) We have $\Ext^1_{\Gamma^{\op}}(\Hom_R(P,R),\Hom_R(U,R))=0$.
Since we have a duality $\Hom_R(\ ,R):\CM(\Gamma)\leftrightarrow\CM(\Gamma^{\op})$, it holds $\Ext^1_\Gamma(U,P)=0$.
There exists an exact sequence $0\to P\to U_0\to U_1\to0$
with $U_i\in\add U$ \cite[Lem. III.2.3]{H}, which must split since $\Ext^1_\Gamma(U,P)=0$. Thus we have $P\in\add U$.
\end{proof}

\medskip
Finally, let us recall the following relation between cluster tilting and
tilting (see \cite[Th. 5.3.2]{I2} for (a), and (b) is clear).

\begin{proposition}\label{XAL}
Let $R=S/(f)$ be a one-dimensional reduced hypersurface singularity and $M$, $N$ and $N'$ cluster
tilting objects in $\CM(R)$.
\begin{itemize}
\item[(a)] $\Hom_R(M,N)$ is a tilting $\End_R(M)$-module of projective
dimension at most one.
\item[(b)] If $N'$ is a cluster tilting mutation of $N$, then $\Hom_R(M,N')$ is a
tilting mutation of $\Hom_R(M,N)$.
\end{itemize}
\end{proposition}

Now we shall prove Theorem \ref{XAB}.
Fix $w\in{\mathfrak S}_n$ and put $\Gamma:=\End_R(M_w)$.
Since $M_w$ is a generator of $R$, the functor $\Hom_R(M_w,\ ):\CM(R)\to\CM(\Gamma)$ is fully faithful.
By Theorem \ref{XAA}, $M_w$ is a cluster tilting object in $\CM(R)$.
By Proposition \ref{XAL}(a), $\Hom_R(M_w,M_{w'})$ ($w'\in{\mathfrak S}_n$) is a
Cohen-Macaulay tilting $\Gamma$-module.

(b) Take any Cohen-Macaulay tilting $\Gamma$-module $U$.
Since $P:=\Hom_R(M_w,R)$ is a projective $\Gamma$-module such that
$\Hom_R(P,R)=M_w=\Hom_R(R,M_w)$ is a projective $\Gamma$-module,
we have $P\in\add U$ by Lemma \ref{XAK}(b).
In particular, by Proposition \ref{XAJ}(a)(b),
\begin{itemize}
\item any Cohen-Macaulay tilting $\Gamma$-module has at most
$(n-1)$ tilting mutations which are Cohen-Macaulay.
\end{itemize}
Conversely, by Proposition \ref{XAI} and Proposition \ref{XAL}(b),
\begin{itemize}
\item any Cohen-Macaulay tilting $\Gamma$-module of the form
$\Hom_R(M_w,M_{w'})$ ($w'\in{\mathfrak S}_n$) has precisely
$(n-1)$ tilting mutations $\Hom_R(M_w,M_{w's_i})$ ($1\le i<n$) which are Cohen-Macaulay.
\end{itemize}
Consequently, any successive tilting mutation of $\Gamma=\Hom_R(M_w,M_w)$ has the form 
$\Hom_R(M_w,M_{w'})$ for some $w'\in{\mathfrak S}_n$ if each step is Cohen-Macaulay.

Using this observation, we shall show that $U$ is isomorphic to $\Hom_R(M_w,M_{w'})$ for
some $w'$. 
Since $\Gamma\ge U$, there exists a sequence
\[\Gamma=T_0>T_1>T_2>\cdots>U\]
satisfying the conditions in Proposition \ref{XAJ}(c).
By Lemma \ref{XAK}(a), each $T_i$ is Cohen-Macaulay.
Thus the above observation implies that each $T_i$ has the form 
$\Hom_R(M_w,M_{w_i})$ for some $w_i\in{\mathfrak S}_n$.
Moreover, $w_i\neq w_j$ for $i\neq j$.
Since ${\mathfrak S}_n$ is a finite group, the above sequence must be finite.
Thus $U=T_i$ holds for some $i$, hence the proof is completed.

(a) Let $U$ be a cluster tilting object in $\CM(R)$.
Again by Proposition \ref{XAL}(a), $\Hom_R(M_w,U)$ is a Cohen-Macaulay
tilting $\Gamma$-module. By part (b) which we already proved,
$\Hom_R(M_w,U)$ is isomorphic to
$\Hom_R(M_w,M_{w'})$ for some $w'\in{\mathfrak S}_n$.
Since the functor $\Hom_R(M_w,\ ):\CM(R)\to\CM(\Gamma)$ is fully faithful,
$U$ is isomorphic to $M_{w'}$, and the former assertion is proved.

For the latter assertion, we only have to show that any rigid object
in $\CM(R)$ is a direct summand of some cluster tilting object in $\CM(R)$.
This is valid by the following general result in \cite[Th. 1.9]{BIRS}.
\qed

\begin{proposition}
Let $\mathcal{C}$ be a 2-CY Frobenius category with a cluster tilting object.
Then any rigid object in $\mathcal{C}$ is a direct summand of some
cluster tilting object in $\mathcal{C}$.
\end{proposition}

\medskip
We end this section with the following application to dimension three.

Now let $S'':=k[[x,y,u,v]]$, $f_i\in{\mathfrak m}=(x,y)$ ($1\le i\le n$)
and $R'':=S''/(f_1\cdots f_n+uv)$.
For $w\in{\mathfrak S}_n$ and $I\subseteq\{1,\cdots,n\}$, we put
\[U^w_i:=(u,f_{w(1)}\cdots f_{w(i)})\subset R'',\ M_w:=\bigoplus_{i=1}^nU^w_i\
\mbox{ and }\ U_I:=(u,\prod_{i\in I}f_i)\subset S''.\]
We have the following result (see \ref{D:vdB} for definition).

\begin{corollary}\label{XAQ}
Under the assumption (A), we have the following.
\begin{itemize}
\item[(a)] There are exactly $n!$ basic cluster tilting objects $M_w$
($w\in{\mathfrak S}_n$) and exactly $2^n-1$ indecomposable rigid objects
$U_I$ ($\emptyset\neq I\subset\{1,\cdots,n\}$) in $\CM(R'')$.
\item[(b)] There are non-commutative crepant resolutions
$\End_{R''}(M_w)$ ($w\in{\mathfrak S}_n$) of $R''$, which are
derived equivalent.
\end{itemize}
\end{corollary}

\begin{proof}
(a) We only have to apply Kn\"orrer periodicity $\ul{\CM}(R)\to\ul{\CM}(R'')$ \cite{Kn,S} as follows:

Since $S^w_i\in\CM(R)$ has a projective resolution
$$
R\stackrel{b}{\longrightarrow}R\stackrel{a}{\longrightarrow}R\to S^w_i\to0
$$
for $a:=f_{w(1)}\cdots f_{w(i)}$ and $b:=f_{w(i+1)}\cdots f_{w(n)}$,
the corresponding object $X\in\CM(R'')$ has a projective resolution
$$
R''^2\stackrel{{u\ \ a\choose b\ -v}}{\longrightarrow}R''^2\stackrel{{v\ \ a\choose b\ -u}}{\longrightarrow}R''^2\to X\to0.
$$
It is easily checked that $X$ is isomorphic to $(u,a)=U^w_i$.

(b) Any cluster tilting object gives a non-commutative crepant
resolution by \cite[Th. 5.2.1]{I2}. They are derived equivalent by \cite[Cor. 5.3.3]{I2}.
\end{proof}

\medskip
For example, 
\[k[[x,y,u,v]]/((x-\lambda_1 y)\cdots(x-\lambda_n y)+uv)\]
has a non-commutative crepant resolution for distinct elements $\lambda_1,\cdots,\lambda_n\in k$.

\section{Link with birational geometry}\label{sec3}
There is another approach to the investigation of cluster tilting
objects for maximal Cohen-Macaulay modules, using birational
geometry. 
More specifically there is a close connection between resolutions of 
three-dimensional Gorenstein singularities and cluster tilting theory,  provided by the so-called 
non-commutative crepant resolutions of Van den Bergh.
This gives at the same time alternative proofs for geometric
results, using cluster tilting
objects. The aim of this section is to establish a link with small
resolutions. We give relevant criteria for having
small resolutions, and apply them to give an alternative approach to
most of the results in the previous sections. 

Let $(R,{\mathfrak m})$ be a three-dimensional complete normal Gorenstein singularity
over an algebraically closed field $k$ of characteristic zero, and let $X = \Spec(R)$. 
 A resolution of singularities $Y \stackrel{\pi}\lar X$ is called 

\begin{itemize}
\item  {\it crepant}, 
if $\omega_Y \cong \pi^* \omega_X$ for canonical sheaves $\omega_X$ and $\omega_Y$ of $X$ and $Y$ respectively.
\item {\it small}, 
if the fibre of the closed point has dimension at most one.
\end{itemize}

\noindent
A small resolution is automatically crepant, but the converse is in general not true. However, both types of resolutions coincide  for  certain important classes of three-dimensional singularities.

A {\it cDV (compound Du Val)} singularity is a three-dimensional singularity
given by the equation 
$$
f(x,y,z) + t g(x,y,z,t)=0,
$$
where $f(x,y,z)$ defines a simple surface singularity and $g(x,y,z,t)$ is arbitrary. 

A cDV singularity is called $cA_n$ 
if the intersection of $f(x,y,z)+tg(x,y,z,t)=0$ with a generic 
hyperplane $ax+by+cz+dt=0$ in $k^4$ is an $A_n$
surface singularity.
{\it Generic} means that the coefficients $(a,b,c,d)$
belong to a non-empty Zariski open subset in $k^4$.

\begin{theorem}\cite[Cor. 1.12, Th. 1.14]{Re}\label{t3.1}
\noindent
Let $X$ be a three-dimensional Gorenstein singularity.
\begin{itemize}
\item[(a)] If $X$ has a small resolution, then it is cDV. 
\item[(b)] If $X$ is an isolated cDV singularity, then any crepant
  resolution of $X$ is small. 
\end{itemize} 
\end{theorem}




Since any isolated cDV singularity is terminal \cite{Re}, we can apply Van den Bergh's
results on non-commutative crepant resolutions defined as follows.

\begin{definition}\cite[Def. 4.1]{V2}\label{D:vdB}
Let $(R,{\mathfrak m})$ be a three-dimensional normal Gorenstein domain.  An $R$-module $M$ gives rise to a non-commutative crepant resolution if
\begin{itemize}
\item[(i)] $M$ is reflexive,
\item[(ii)] $A = \End_R(M)$ is Cohen-Macaulay as an $R$--module,
\item[(iii)] $\mathrm{gl.dim}(A) = 3$. 
\end{itemize}
\end{definition}

The following result establishes a useful connection.
\begin{theorem}\cite[Cor. 3.2.11]{V1}\cite[Th. 6.6.3]{V2}
Let $(R,{\mathfrak m})$ be an isolated cDV singularity.
Then there exists a crepant resolution of $X=\Spec(R)$ if and only
if there exists a non-commutative one in the sense of Definition \ref{D:vdB}.
\end{theorem}

\noindent
The existence of a non-commutative crepant resolution turns out to be
equivalent to the existence of a cluster tilting object in the
triangulated category $\underline{\CM}(R)$.

\begin{theorem}\cite[Th. 5.2.1]{I2}\cite[Cor. 8.13]{IR}\label{IR}
Let $(R,{\mathfrak m})$ be a three-dimensional normal Gorenstein domain
which is an isolated singularity.
Then the existence of a non-commutative crepant
resolution is equivalent to the existence of a cluster tilting object
in the stable category of maximal Cohen-Macaulay modules
$\underline{\CM}(R)$. 
\end{theorem}

\begin{proof}
For convenience of the reader, we give an outline of the proof (see also Proposition \ref{XAE}).

Let us first assume that $M$ is a cluster tilting object in $\underline{\CM}(R)$. Then $M$ is automatically reflexive.  From the exact sequence
$$
 0 \lar \Omega(M) \lar F \lar M \lar 0
$$
we obtain
\begin{equation}\label{E:resol}
0 \lar \End_R(M) \lar \Hom_R(F, M) \lar \Hom_R(\Omega(M), M) \lar \Ext^1_R(M,M)  \lar 0.
\end{equation}
Since $M$ is rigid, $\Ext^1_R(M, M) = 0$. Moreover, 
$\dpth(\Hom_R(F, M)) = \dpth(M) = 3$ and

\noindent
$\dpth(\Hom_R(\Omega(M), M)
\ge 2$, and hence $\dpth(\End_R(M)) = 3$ and $A=\End(M)$ is maximal
Cohen-Macaulay over $R$.

For the difficult part of this implication, claiming that $\mathrm{gl.dim}(A) = 3$, we refer to \cite[Th. 3.6.2]{I1}. 

For the other direction, let $M$ be a module giving rise to a non-commutative crepant resolution. Then by  \cite[Th. 8.9]{IR} there exists another module $M'$ giving rise to a non-commutative crepant resolution,  which is maximal Cohen-Macaulay  and contains  $R$ as a direct summand.  

By the assumption, $\dpth(\End_R(M')) = 3$ and we can apply
\cite[Lem. 8.5]{IR} to the exact sequence (\ref{E:resol})  to deduce
that  $\Ext^1_R(M', M') = 0$, so that $M'$ is rigid. The difficult
part saying that $M'$ is cluster tilting is proven in   \cite[Th. 5.2.1]{I2}.
\end{proof}

\medskip
We now summarize the results of this section.
\begin{theorem}\label{4.5}
  Let $(R,{\mathfrak m})$ be an isolated cDV singularity. Then the following are equivalent.
\begin{itemize}
\item[(a)] $\Spec(R)$ has a small resolution.

\item[(b)] $\Spec(R)$ has a crepant resolution.

\item[(c)] $(R,{\mathfrak m})$ has a non-commutative crepant resolution.

\item[(d)] $\ul{\CM}(R)$ has a cluster tilting object.
\end{itemize}
\end{theorem}

We have an efficient criterion for existence of a small resolution of a $cA_n$--singularity.

\begin{theorem}\cite[Th. 1.1]{Kat}\label{4.6}
Let $X = \Spec(R)$ be an isolated $cA_n$--singularity.
\begin{itemize}
\item[(a)] Let $Y \lar X$ be  a small resolution.
Then the exceptional curve in $Y$ is a chain of $n$ projective
lines and $X$ has the  form $g(x,y)+uv$, where the curve singularity
$g(x,y)$ has $n+1$  distinct branches at the origin. 

\item[(b)] If $X$ has the  form $g(x,y)+uv$, where the curve singularity
$g(x,y)$ has $n+1$  distinct branches at the origin, then 
$X$ has a small resolution. 
\end{itemize}
\end{theorem}

Using the criterion of Katz together with Kn\"orrer periodicity, we
get additional equivalent conditions in a special case. 

\begin{theorem}\label{4.7}
  Let $(R,{\mathfrak m})$ be an isolated $cA_n$--singularity
  defined by the equation $g(x,y)+zt$. Then the following conditions
  are equivalent in addition to (a)-(d) in Theorem \ref{4.5}.
\begin{itemize}
\item[(e)] Let $R'$ be a one-dimensional singularity defined by $g(x,y)$. Then $\ul{\CM}(R')$ has a cluster tilting object.

\item[(f)] The number of irreducible power series in the prime decomposition of $g(x,y)$ is $n+1$.
\end{itemize}
\end{theorem}

\begin{proof}
(a)$\Leftrightarrow$(f) This follows from Theorem \ref{4.6}.

(d)$\Leftrightarrow$(e) By the Kn\"orrer periodicity there is an
equivalence of triangulated categories between the stable categories 
$\underline{\CM}(R) \cong \underline{\CM}(R')$.
For, the equivalence of these stable categories given in \cite{Kn,S} is induced by an exact functor taking projectives to projectives.
\end{proof}

\begin{theorem}\label{4.8}
  Assume that the equivalent conditions in Theorem \ref{4.7} are satisfied. Then the following numbers are equal.
\begin{itemize}
  \item[(a)] One plus the number of irreducible components of the exceptional curve of a small resolution of $\Spec(R)$.
\item[(b)] The number of irreducible power series in the prime decomposition of $g(z,t)$.
\item[(c)] The number of simple  modules of non-commutative crepant resolutions of $(R,{\mathfrak m})$.
\item[(d)] One plus the number of non-isomorphic indecomposable summands of
  basic cluster tilting objects in $\ul{\CM}(R)$.
\end{itemize}
\end{theorem}

\begin{proof}
(a) and (b) are equal by Theorem \ref{4.6}.

(a) and (c) are equal by \cite[Th. 3.5.6]{V1}.

(c) and (d) are equal by \cite[Cor. 8.8]{IR}.
\end{proof}

\section{Application to curve singularities}\label{sec4}


%
%

In this section we apply results in the previous section to some curve
singularities to investigate whether they have some cluster tilting object or
not. In addition to simple singularities, we study some other nice singularities.
In what follows we refer to \cite{AGV} as a general reference for classification 
of singularities.    

To apply results in previous sections to minimally elliptic
singularities, we also consider a three-dimensional hypersurface singularity 
$$
T_{p,q,2,2}(\lambda) = k[[x,y,u,v]]/(x^p + y^q + \lambda x^2 y^2 + uv).
$$

%
%
%

\medskip
To apply Theorem \ref{4.7} to a curve singularity, we have to know
that the corresponding three-dimensional singularity is $cA_n$.
It is given by the following result, where
we denote by $\ord(g)$ the degree of the lowest term of a power series
$g$. 

\begin{proposition}\label{5.1}
We have the following properties of three-dimensional 
hypersurface singularities: 
\begin{itemize}
\item[(a)] $A_n$ ($1\le n$) is a $cA_1$--singularity,
\item[(b)] $D_n$ ($4\le n$) and $E_n$ ($n=6,7,8$) are $cA_2$--singularities,
\item[(c)] $T_{3,q,2,2}(\lambda)$ ($6\le q$) is a $cA_2$--singularity,
\item[(d)] $T_{p,q,2,2}(\lambda)$ ($4\le p\le q$) is a $cA_3$--singularity, 
\item[(e)] $k[[x,y,z,t]]/(x^2+y^2+g(z,t))$ ($g\in k[[z,t]]$) is a
  $cA_m$--singularity if $m=\ord(g)-1\ge1$.
\end{itemize}
\end{proposition}

We shall give a detailed proof at the end of this section.
In view of Theorem \ref{4.7} and Proposition \ref{5.1}, we have the
following main result in this section.

\begin{theorem}
\begin{itemize}
\item[(a)] A simple three-dimensional singularity satisfies the
  equivalent conditions in Theorem \ref{4.7} if and only if it is of
  type $A_n$ ($n$ is odd) or $D_n$ ($n$ is even).
\item[(b)] A $T_{p,q,2,2}(\lambda)$--singularity satisfies the
  equivalent conditions in Theorem \ref{4.7} if and only if 
$p = 3$ and $q$ is even or if both $p$ and $q$ are even.
\item[(c)] A singularity $k[[x,y,u,v]]/(uv+f_1\cdots f_n)$ with
irreducible and mutually prime $f_i\in(x,y)\subset k[[x,y]]$ ($1\le i\le n$) satisfies the
equivalent conditions in Theorem \ref{4.7} if and only if 
$f_i\notin(x,y)^2$ for any $i$.
\end{itemize}
\end{theorem}

\begin{proof}
Each singularity is $cA_m$ by Proposition \ref{5.1}, and defined by an equation
of the form $g(x,y)+uv$. By Theorem \ref{4.7}, we only have to check
whether the number of irreducible power series factors of $g(x,y)$ is $m+1$ or not.

(a) For an $A_n$--singularity, we have $m=1$ and $g(x,y)=x^2+y^{n+1}$.
So $g$ has two factors if and only if $n$ is odd.

For a $D_n$--singularity, we have $m=2$ and $g(x,y)=(x^2+y^{n-2})y$.
So $g$ has three factors if and only if $n$ is even.

For an $E_n$--singularity, we have $m=2$ and $g(x,y)=x^3+y^4$,
$x(x^2+y^3)$ or $x^3+y^5$. In each case, $g$ does not have three factors.

(b) First we consider the simply elliptic case. We have $m=2$ and
$g(x,y)=y(y-x^2)(y-\lambda x^2)$ for $(p,q)=(3,6)$,
and $m=3$ and $g(x,y)=xy(x-y)(x-\lambda y)$ for $(p,q)=(4,4)$.
In both cases, $g$ has $m+1$ factors.

Now we consider the cusp case.
We have $m=2$ for $p=3$ and $m=3$ for $p>3$, and $g(x,y)=(x^{p-2}-y^2)(x^2-y^{q-2})$.
So $g$ has $m+1$ factors if and only if $p=3$ and $q$ is even or if
both $p$ and $q$ are even. 

(c) We have $m=\sum_{i=1}^n\ord(f_i)-1$ and $g=f_1\cdots f_n$. So $g$ 
has $m+1$ factors if and only if $\ord(f_i)=1$ for any $i$.
\end{proof}

Immediately we have the following conclusion.

\begin{corollary}\label{5.3}
\begin{itemize}
\item[(a)] A simple curve singularity $R$
has a cluster tilting object if and only if it is of
type $A_n$ ($n$ is odd) or $D_n$ ($n$ is even).
The number of non-isomorphic indecomposable summands
of basic cluster tilting objects in $\ul{\CM}(R)$ is $1$ for type $A_n$ ($n$ is odd) and $2$
for type $D_n$ ($n$ is even).
\item[(b)] A $T_{p,q}(\lambda)$-singularity $R$ has a cluster tilting
object if and only if 
$p = 3$ and $q$ is even or if both $p$ and $q$ are even.
The number of non-isomorphic indecomposable summands
of basic cluster tilting objects in $\ul{\CM}(R)$ is $2$ if $p=3$ and $q$ is even, and $3$ if
both $p$ and $p$ are even.
\item[(c)] A singularity $R=k[[x,y]]/(f_1\cdots f_n)$ with
irreducible and mutually prime $f_i\in(x,y)\subset k[[x,y]]$ ($1\le i\le n$) 
has a cluster tilting object if and only if $f_i\notin(x,y)^2$ for any $i$.
In this case, the number of non-isomorphic indecomposable summands
of basic cluster tilting objects in $\ul{\CM}(R)$ is $n-1$.
\end{itemize}
\end{corollary}

In view of Theorem \ref{XAB}, we have completed the proof of Theorem \ref{main3}.

In the rest of this section, we shall prove Proposition \ref{5.1}.

Let $k$ be an algebraically closed field of characteristic zero, 
$R = k[[x_1,x_2,\dots,x_n]]$ the local ring of formal power series
and ${\mathfrak m}$ its maximal ideal. 
We shall need the following standard notions.

\begin{definition}
For $f \in {\mathfrak m}^2$ we 
denote by $J(f) = \langle \frac{\partial f}{\partial x_1},\dots,
\frac{\partial f}{\partial x_n} \rangle$ its Jacobi ideal. 
The Milnor number $\mu(f)$ is defined as
$$
\mu(f):= \dim_k(R/J(f)).
$$
\end{definition}

\noindent
The following lemma is standard (see for example \cite{AGV, GLSh}):

\begin{lemma}
A  hypersurface singularity $f=0$ is isolated if and only if $\mu(f) < \infty$. 
\end{lemma}

\begin{definition}[\cite{AGV}]
Two hypersurface singularities $f= 0$ and $g= 0$ are called
\emph{right} equivalent ($f \stackrel{r}\sim g$)  if
there exists an algebra automorphism $\varphi \in Aut(R)$ such that 
$g = \varphi(f)$. 
\end{definition}

Note that $f \stackrel{r}\sim g$ implies an isomorphism of $k$--algebras
$$R/(f) \cong R/(g).$$ The following lemma is straightforward, see
for example  \cite[Lem. 2.10]{GLSh}.  

\begin{lemma}\label{milnor}
Assume $f\stackrel{r}\sim g$,  then $\mu(f) = \mu(g)$. 
\end{lemma}

In what follows, we shall need the 
next  standard result on classification of singularities, see for example
\cite[Cor. 2.24]{GLSh}.

\begin{theorem}\label{T:crit2}
Let $f \in {\mathfrak m}^2$ be an isolated singularity with Milnor number 
$\mu$. Then 
$$
f \stackrel{r}\sim  f + g
$$
for any  $g \in {\mathfrak m}^{\mu + 2}$. 
\end{theorem}

We shall need the following easy lemma.

\begin{lemma}\label{L:simple}
Let $f = x^2 + y^2 + p(x,y,z)$, where
$$p(x,y,z) = z^n + p_1(x,y) z^{n-1} + \dots + p_n(x,y)$$ is a homogeneous
form of degree $n \ge 3$. Then 
$$
f \stackrel{r}\sim   x^2 + y^2 + z^n.
$$ 
\end{lemma}

\noindent
\emph{Proof}. Write $p(x,y,z) = z^n + xu + yv$ for some homogeneous
forms $u$ and $v$ of degree $n-1$. Then 
$$
x^2 + y^2 + z^n + xu + yv = (x + u/2)^2 + (y + v/2)^2 + z^n - (u^2 + v^2)/4.
$$
After a change of variables 
$x \mapsto  x + u/2$, $y \mapsto y + v/2$ and $z \mapsto z$
we reduce $f$ to the form 
$$
f = x^2 + y^2 + z^n + h,
$$
where $h \in {\mathfrak m}^{2(n-1)} \subset {\mathfrak m}^{n+1}$. 
Note that $\mu(x^2 + y^2 + z^n) = n-1$, hence by Theorem 
\ref{T:crit2} we have
$$
f \stackrel{r}\sim   x^2 + y^2 + z^n.
$$
\qed

Now we are ready to give a proof of  Proposition \ref{5.1}.
We only have to show the assertion (e) since the other cases are special
cases of this. We denote by $H$ the hyperplane in a
four-dimensional space defined by the equation $t  = \alpha x + \beta
y + \gamma z$, $\alpha, \beta, \gamma \in k$. 
We put
$$g(z,t)=a_0z^{m+1}+a_1z^mt+\cdots+a_{m+1}t^{m+1}+(\mbox{higher terms}).$$
Then the intersection of $H$ with the singularity defined by the
equation $x^2+y^2+g(z,t)$ is given by the equation
$f=h+(\mbox{higher terms})$, where 
$$h=x^2+y^2+a_0z^{m+1}+a_1z^m(\alpha x+\beta y+\gamma
z)+\cdots+a_{m+1}(\alpha x+\beta y+\gamma z)^{m+1}.$$

Now we consider the case $m=1$.
We have $h\stackrel{r}\sim x^2+y^2+z^2$ since any 
quadratic form can be diagonalized using linear transformations. 
By Lemma \ref{milnor}, we have $\mu(h)=\mu(x^2+y^2+z^2)=1$.
Hence $f\stackrel{r}\sim h\stackrel{r}\sim x^2+y^2+z^2$ by Theorem
\ref{T:crit2}. 

Next we consider the case $m\ge 2$. 
Assume $\alpha\in k$ satisfies $a_0+a_1\alpha+\cdots+a_{m+1}\alpha^{m+1}\neq0$.
By Lemma \ref{L:simple}, we have $h\stackrel{r}\sim x^2+y^2+z^{m+1}$.
By Lemma \ref{milnor}, we have $\mu(h)=\mu(x^2+y^2+z^{m+1})=m$.
Hence $f\stackrel{r}\sim h\stackrel{r}\sim
x^2+y^2+z^{m+1}$ by Theorem \ref{T:crit2}.

Consequently, $x^2+y^2+g(z,t)$ is $cA_m$.
\qed

\section{Examples of 2-CY tilted algebras}\label{CYtilted}

Since the 2-CY tilted algebras coming from maximal Cohen-Macaulay modules over hypersurfaces have some nice properties, it is of interest to have more explicit information about such algebras.
This section is devoted to some such computations for algebras coming
from minimally elliptic singularities. We obtain algebras appearing in
classification lists for some classes of tame self-injective algebras
\cite{Er,BS,Sk}.

We start with giving some general properties which are direct consequences of Lemma \ref{L1.2}.
\begin{theorem}
  Let $(R,{\mathfrak m})$ be an odd-dimensional isolated hypersurface singularity and $\ga$ a 2-CY tilted algebra coming from $\ul{\CM}(R)$. Then we have the following.
\begin{itemize}
  \item[(a)] $\ga$ is a symmetric algebra.
  \item[(b)] All components in the stable AR-quiver of infinite type $\ga$ are tubes of rank 1 or 2.
\end{itemize}
\end{theorem}
We now start with our computations of 2-CY tilted algebras coming from
minimally elliptic singularities. We first introduce and investigate
two classes of algebras, and then show that they are isomorphic to
2-CY tilted algebras coming from minimally elliptic singularities.

For a quiver $Q$ with finitely many vertices and arrows
we define the  radical completion $\widehat{kQ}$ of the path algebra $kQ$ by the formula
$$
\widehat{kQ} = \lim_{\longleftarrow}  kQ/\rad^n(kQ).
$$

The reason we deal with completion is the following:
Let $Q$ be a finite quiver, $J$ the ideal of $\widehat{kQ}$ generated by the arrows and $I\subseteq J^2$ a complete ideal such that $\la=\widehat{kQ}/I$ is finite-dimensional.
\begin{lemma}
  The ideal $I$ is generated in $\widehat{kQ}$ by a minimal system
  of relations, that is, a set of elements $\rho_1, \cdots, \rho_n$
  of $I$ whose images form a $k$-basis of $I/{IJ+JI}$. 
\end{lemma}
The lemma is shown by a standard argument (cf \cite[Section
3]{BMR3}). Its analogue for the non complete path algebra is not
always true.  For example, for the algebra $\la=B_{2,2}(\lambda)$
defined below, the elements $\rho_1, \cdots, \rho_n$ listed as
generators for $I$ form a minimal system of relations. So they
generate $I$ in $\widehat{kQ}$. They also yield a $k$-basis of
$I'/{I'J+JI'}\xto{\sim}I/{IJ+JI}$, where $I'=I\cap kQ$ and $J'=J\cap
kQ$. But they do not generate the ideal $I'$ of $kQ$ since, as one can
show, the quotient $kQ/\langle\rho_1, \cdots, \rho_n\rangle$ is
infinite-dimensional.

On the other hand, the ideal $I'$ is generated by the preimage
$\rho_1, \cdots, \rho_n$ of a basis of $I'/{I'J'+J'I'}$ if the
quotient $kQ/\langle\rho_1, \cdots, \rho_n\rangle$ is
finite-dimensional, since then the ideal $\langle\rho_1, \cdots, \rho_n\rangle$ 
contains a power of $J'$. This happens for example for the algebra
$A_2(\lambda)$ as defined below, cf. also \cite[5.9]{Sk} and \cite[Th. 1]{BS}.

We know that for all vertices $i,j$ of $Q$, we have 
$$
\dim_k e_i(I/{IJ+JI})e_j=\dim_k \Ext_{\la}^2(S_i,S_j)
$$
where $S_i$ and $S_j$ denote the simple $\la$-modules corresponding to the vertices $i$ and $j$ \cite{B}. When $\la$ is 2-CY tilted, then
$$
\dim \Ext^1_{\la}(S_j,S_i)\ge\dim\Ext^2_{\la}(S_i,S_j)
$$
(see \cite{BMR3,KR}). Thus the number of arrows in $Q$ is an
upper bound on the number of elements in a minimal system of
relations. 

\begin{definition} (1) For $q \ge 2$ and $\lambda \in k^*$ we write
$A_q(\lambda) = \widehat{kQ}/I$, where
$$
Q = 
\xymatrix{
    \cdot \ar@(dl,ul)[]^{\varphi} \ar@<0.4ex>[r]^{\alpha}&
    \cdot\ar@<0.4ex>[l]^{\beta}  \ar@(dr,ur)[]_{\psi}\\ 
    }
$$ 
and 
$$
I = \langle \psi  \alpha - \alpha \varphi, \beta \psi  - \varphi  \beta,  
\varphi^2 -  \beta \alpha, 
\psi^q  - \lambda \alpha \beta \rangle.
$$
If $q = 2$, then  we additionally assume $\lambda \ne 1$. (It can be shown that 
for $q \ge 3$ we have $A_q(\lambda) \cong A_q(1)$, so we drop the parameter $\lambda$ 
in this case.)

\medskip
\noindent
(2) For $p, q \ge 1$ and $\lambda \in k^*$ we write $B_{p,q}(\lambda) = 
\widehat{kQ}/I$, where
$$
Q = 
\xymatrix{
    \cdot \ar@(dl,ul)[]^{\varphi}   \ar@<0.4ex>[r]^{\alpha}&
    \cdot\ar@<0.4ex>[l]^{\beta}  \ar@<0.4ex>[r]^{\gamma} & 
\cdot\ar@<0.4ex>[l]^{\delta} \ar@(dr,ur)[]_{\psi}   \\ 
    }
$$ 
and 
$$
I = \langle 
\beta \alpha - \varphi^p, \gamma \delta - \lambda \psi^q, 
\alpha \varphi - \delta \gamma \alpha, \varphi \beta - \beta \delta \gamma, 
\delta \psi  - \alpha \beta \delta, \psi \gamma - \gamma \alpha \beta
\rangle.
$$
For $p = q = 1$ we additionally assume $\lambda \ne 1$. 
\end{definition}
When $p=q=1$, the generators $\varphi$  and $ \psi$ can be excluded and $B_{1,1}(\lambda)$ is given by 
the completion of the path algebra of the quiver 
$$
Q = 
\xymatrix{
    \cdot  \ar@<0.4ex>[r]^{\alpha}&
    \cdot\ar@<0.4ex>[l]^{\beta}  \ar@<0.4ex>[r]^{\gamma} & 
\cdot\ar@<0.4ex>[l]^{\delta}   \\ 
    }
$$ 
modulo the relations  
$$
I = \langle 
\alpha \beta \alpha - \delta \gamma \alpha, 
\alpha \beta \delta - \lambda \delta \gamma \delta, 
\gamma \alpha \beta - \lambda \gamma \delta \gamma, 
\beta \delta \gamma  - \beta \alpha \beta \rangle.
$$
For $(p,q) \ne (1,1)$ we have $B_{p,q}(\lambda) \cong B_{p,q}(1)$. In particular,
for $p = 1$ and   $q \ge 2$ the algebra is isomorphic to $\widehat{kQ}/I$, where 
$$
Q = 
\xymatrix{
    \cdot  \ar@<0.4ex>[r]^{\alpha}&
    \cdot\ar@<0.4ex>[l]^{\beta}  \ar@<0.4ex>[r]^{\gamma} & 
\cdot\ar@<0.4ex>[l]^{\delta} \ar@(dr,ur)[]_{\psi}  \\ 
    }
$$ 
and 
$$
I = \langle \gamma \delta - \psi^q, \alpha \beta \alpha - \delta \gamma \alpha, 
\beta \alpha \beta - \beta \delta \gamma, \delta \psi - \alpha \beta \delta, 
\psi \gamma - \gamma \alpha \beta  \rangle.  
$$

It turns out that the algebras $A_q(\lambda)$ and $B_{p,q}(\lambda)$ are finite dimensional. 
In order to show this it suffices to check that all  oriented cycles in 
$\widehat{kQ}/I$ are nilpotent. 

\vspace{0.5cm}

\begin{lemma}
In the algebra $A_q(\lambda)$ the following zero relations hold:
$$
\alpha \beta \alpha = 0, \beta \alpha \beta = 0, \alpha \varphi^2 = \psi^2 \alpha = 0,
\varphi^2 \beta = \beta \psi^2 = 0, \varphi^4 = 0, \psi^{q+2} = 0.
$$
\end{lemma} 

\noindent
\textit{Proof}. We have to consider separately the cases $q = 2$ and $q \ge 3$. 

\medskip
\noindent
Let $q = 2$, then we assumed $\lambda \ne 1$. We have 
$$
\alpha \beta \alpha = \alpha \varphi^2 = \psi^2 \alpha = \lambda^{-1} 
\alpha \beta \alpha,
$$
hence $\alpha \beta \alpha = 0.$ In a similar way we obtain $\beta \alpha \beta = 0$. 
Then $\alpha \varphi^2 = \alpha \beta \alpha = 0$, $\varphi^2 = \beta \alpha \beta \alpha = 0$
and the remaining zero relations follow analogously. 

\medskip
\noindent
Let $q \ge 3$. Then 
$$
\psi^q \alpha = \alpha \beta \alpha = \alpha \varphi^2 = \psi^2 \alpha,
$$
so $(1 - \psi^{q-2})\psi^2 \alpha = 0$ and hence 
$$\psi^2 \alpha = \alpha \beta \alpha = 0$$ 
in  $\widehat{kQ}/I$. The remaining zero relations follow similarly. 
\qed

\vspace{0.5cm}

\begin{lemma}
We have the following relations in 
$B_{p,q}(\lambda)$:
$$
\varphi^{p+2} = 0, \psi^{q+2} = 0, 
\gamma \alpha \varphi =  \psi \gamma \alpha = 0,
\varphi \beta \delta  = \beta \delta  \psi =0. 
$$
Moreover, $\alpha \beta \cdot \delta \gamma = \delta \gamma \cdot \alpha \beta$. 
For $q \ge p \ge 2$ we have
$$
(\alpha \beta)^2 = (\delta \gamma)^2 = 0,
$$
for  $q > p = 1$ we have
$$
(\alpha \beta)^3 = 0, (\delta \gamma)^2 = 0, (\alpha \beta)^2 \cdot (\delta \gamma) = 0
$$
and for $p = q = 1$ 
$$
(\alpha \beta)^3 = (\gamma \delta)^3 = 0, (\alpha \beta)^2 = \alpha \beta \cdot \delta \gamma = 
\lambda (\delta \gamma)^2. 
$$
\end{lemma}

\noindent
The proof is completely parallel to the proof of the previous lemma and is therefore skipped. 
\qed

\medskip
\medskip
\noindent
The main result of this section is the following

\begin{theorem}\label{p4.6}
(a) Let $R$ be a $T_{3, 2q +2}(\lambda)$--singularity, where $q \ge 2$ and $\lambda \in k^*$. 
Then in the triangulated category $\underline{\CM}(R)$ there exists a
cluster tilting object 
with the corresponding 2-CY-tilted algebra isomorphic to $A_q(\lambda)$. 

(b)For $R = T_{2p +2, 2q +2}(\lambda)$ the category
$\underline{\CM}(R)$ has a cluster tilting object with endomorphism algebra isomorphic to
$B_{p,q}(\lambda)$.
\end{theorem}

\noindent
\textit{Proof}.
(a) We consider first the case of $T_{3, 2q +2}(\lambda)$. 
 
\medskip
\noindent
The coordinate ring of  $T_{3,6}(\lambda)$ is isomorphic to 
$$R = k[[x,y]]/(y(y-x^2)(y-\lambda x^2)),$$
 where $\lambda \ne 0,1$.  
Consider Cohen-Macaulay modules $M$ and $N$ 
given by  the two-periodic free resolutions 
$$
\left\{
\begin{array}{l}
M = (R \xrightarrow{y-x^2} R \xrightarrow{y(y-\lambda x^2)} R), \\
N = (R \xrightarrow{y(y-x^2)} R \xrightarrow{y-\lambda x^2} R).
\end{array}
\right.
$$
Then $M\oplus N$ is cluster tilting by Theorem \ref{XAA} or Corollary \ref{5.3}.
In order to compute the endomorphism algebra $\underline\End(M \oplus N)$, note  that 
$$
\underline\End(M) \cong k[\varphi]/\langle\varphi^4\rangle, 
$$
where $\varphi = (x,x)$ is an endomorphism of $M$ viewed as a two-periodic map of a free resolution. In $\underline\End(M)$ we have $(y,y) = (x,x)^2 = \varphi^2$.
Similarly, 
$$
\underline\End(N) \cong k[\psi]/\langle\psi^4\rangle, \text{ }\psi  = (x,x),\text{ } (y,y) = \lambda (x,x)^2 =  
\lambda \psi^2,
$$
and 
$$
\underline\Hom(M,N) = k^2 = \langle (1,y), (x, xy)\rangle, 
\quad 
\underline\Hom(N,M) = k^2 = \langle (y,1), (xy, x)\rangle.
$$
The isomorphism $A_2(\lambda) \lar \underline\End(M \oplus N)$ is given by
$$
\varphi \mapsto  (x,x), \psi\mapsto  (x,x),  \alpha \mapsto (1,y), \beta \mapsto
 (y,1).  
$$

\noindent
Assume  now $q \ge 3$ and $R = T_{3, 2q+2}$. By \cite{AGV} we may write  
$$R = k[[x,y]]/((x-y^2)(x^2 - y^{2q})).$$ 
Consider the Cohen-Macaulay module 
$M \oplus N$, where 
$$
\left\{
\begin{array}{l}
M = (R \xrightarrow{x-y^2} R \xrightarrow{x^2 - y^{2q}} R), \\
N = (R \xrightarrow{(x-y^2)(x+y^q)} R \xrightarrow{x-y^q} R).
\end{array}
\right.
$$
Again, by a straightforward calculation 
$$
\underline\End(M) \cong k[\varphi]/\langle\varphi^4\rangle, \text{ }\varphi = (y,y), 
\quad
\underline\End(N) \cong k[\psi]/\langle\psi^{q+2}\rangle, \text{ }\psi = (y,y)
$$
and 
$$
\underline\Hom(M, N) = k^2 = \langle (1, x+y^q), (y, y(x+y^q))\rangle,
$$
$$
\underline\Hom(M, N) = k^2 = \langle (x+y^q, 1), (y(x+y^q), y)\rangle.
$$
If $q \ge 4$ then $\underline\End(M \oplus N)$ is isomorphic to $\widehat{kQ}/I$, where
$$
Q= 
\xymatrix{
    \cdot \ar@(dl,ul)[]^{\varphi} \ar@<0.4ex>[r]^{\alpha}&
    \cdot\ar@<0.4ex>[l]^{\beta}  \ar@(dr,ur)[]_{\psi}\\ 
    }
$$ 
and the relations are
$$
\beta \alpha = \varphi^2, \alpha \beta = 2 \psi^q, \alpha \varphi = \psi \alpha, 
\varphi \beta = \beta \psi
$$
for 
$$
\varphi = (y,y), \psi = (y,y),  \alpha = (1, x+y^q), \beta = (x+y^q, 1).
$$
By rescaling all generators 
$\alpha \mapsto 2^a \alpha, \beta \mapsto 2^b \beta, \varphi \mapsto 2^f \varphi, 
\psi \mapsto 2^g \psi$ for properly chosen  $a,b,f,g \in \mathbb{Q}$ 
one can easily show $\underline\End(M \oplus N) \cong 
A_q$. 

\medskip
\medskip
The case $q = 3$ has to be considered separately, since this time  the relations are
$$
\beta \alpha = \varphi^2 + \varphi^3, \alpha \beta = 2 \psi^q, \alpha \varphi = \psi \alpha, 
\varphi \beta = \beta \psi.
$$
We claim that there exist invertible power series 
$u(t), v(t), w(t), z(t) \in k[[t]]$ such that the new
generators 
$$
\varphi' = u(\varphi) \varphi, \psi' = v(\psi)\psi, 
\alpha' = \alpha w(\varphi) = w(\psi) \alpha, 
\beta' = \beta z(\psi) = z(\varphi) \beta
$$
satisfy precisely the relations of the algebra  $A_3$. This is fulfilled provided 
we have the following equations  in $k[[t]]$:
$$
\left\{
\begin{array}{l}
zw = u^2(1+tu) \\
zw = 2v^3 \\
uw = vw \\
uz = vz.
\end{array}
\right.
$$
This system is equivalent to 
$$u(t) = v(t) = (2-t)^{-1} = \frac{1}{2}(1 + \frac{t}{2} + (\frac{t}{2})^2 + \dots)
$$
and hence the statement is proven. 

\medskip
\medskip
The case of $T_{2p+2, 2q+2}(\lambda)$ is essentially similar. For $p = q = 1$ we have 
$$R = k[[x,y]]/(xy(x-y)(x-\lambda y)).$$  
Take  
$$
\left\{
\begin{array}{l}
M = (R \xrightarrow{x-y} R \xrightarrow{xy(x-\lambda y)} R), \\
N = (R \xrightarrow{x(x-y)} R \xrightarrow{y(x-\lambda y)} R), \\
K = (R \xrightarrow{xy(x-y)} R \xrightarrow{x-\lambda y} R). \\
\end{array}
\right.
$$
By Theorem \ref{XAA} or Corollary \ref{5.3}, $M \oplus N \oplus K$
is cluster tilting.
Moreover,  $B_{1, 1}(\lambda)\simeq  \underline\End(M\oplus N \oplus K)$.
\medskip
\medskip

\noindent
Let now $$R = k[[x,y]]/((x^p -y)(x^p +y)(y^q -x)(y^q+x)),$$ 
where $(p,q) \ne (1,1)$ and 
$$
\left\{
\begin{array}{l}
M = (R \xrightarrow{x^p-y} R \xrightarrow{(y^q+x)(y^q-x)(x^p+y)} R), \\
N = (R \xrightarrow{(x^p-y)(x^p +y)} R \xrightarrow{(y^q-x)(y^q + x) } R), \\
K = (R \xrightarrow{(x^p-y)(x^p+y)(y^q+x)} R \xrightarrow{y^q-x} R).
\end{array}
\right.
$$
By Theorem \ref{XAA} or Corollary \ref{5.3}, $M \oplus N \oplus K$ is cluster tilting, and by a similar case-by-case analysis 
it can be verified that $\underline\End(M \oplus N \oplus K) \cong B_{p,q}$.
\qed 

\medskip

We have seen that  the algebras $A_q(\lambda)$ and $B_{p,q}(\lambda)$
are symmetric, and the indecomposable nonprojective modules have
$\tau$-period at most 2, hence $\Omega$-period dividing 4 since
$\tau=\Omega^2$ in this case. A direct computation shows that the
Cartan matrix is nonsingular. Note that these algebras appear in
Erdmann's list of algebras of \emph{quaternion type} \cite{Er}, see
also \cite{Sk}, that is, in addition to the above properties, the
algebras are tame. Note that for the corresponding algebras, more
relations are given in Erdmann's list. This has to do with the fact
that we are working with the completion, as discussed earlier. In our
case all relations correspond to different arrows in the quiver. The
simply elliptic ones also appear in Bia\l kowski-Skowro\'nski's list
of weakly symmetric tubular algebras with a nonsingular Cartan
matrix. 

This provides a link between some stable categories of maximal Cohen-Macaulay modules over isolated hypersurface singularities, and some classes of finite dimensional algebras, obtained via cluster tilting theory.

Previously a link between maximal Cohen-Macaulay modules and finite
dimensional algebras was given with the canonical algebras of Ringel,
via the categories ${\rm Coh}(\mathbb{X})$ of coherent sheaves on weighted
projective lines in the sense of Geigle-Lenzing \cite{GL}. Here the
category of vector bundles is equivalent to the category of
graded maximal Cohen-Macaulay modules with degree zero maps, over some
isolated singularity. And the canonical algebras are obtained as
endomorphism algebras of certain tilting objects in ${\rm
  Coh}(\mathbb{X})$ which are vector bundles. 

Note that it is known from work of Dieterich \cite{Di}, Kahn \cite{Ka}, Drozd and  Greuel \cite{DG} that minimally elliptic curve singularities  have tame Cohen-Macaulay representation type. Vice versa, any Cohen-Macaulay tame reduced hypersurface curve singularity is isomorphic to one of the $T_{p,q}(\lambda)$, see \cite{DG}. Moreover, simply elliptic singularities are tame of polynomial growth and cusp singularities are tame of exponential growth. Furthermore, the Auslander-Reiten quiver of the corresponding stable categories  of maximal Cohen-Macaulay modules  consists of tubes of rank one or two, see \cite[Th. 3.1]{Ka} and \cite[Cor. 7.2]{DGK}. 

It should follow from the tameness of $\CM(T_{3,p}(\lambda))$  and $\CM(T_{p,q}(\lambda))$ that the associated 2-CY tilted algebras are tame.

We point out that in the wild case we can obtain symmetric 2-CY tilted
algebras where the stable AR-quiver consists of tubes of rank one and
two, and most of them should be wild. It was previously known that
there are examples of wild selfinjective algebras whose AR-quivers
consist of tubes of rank one or three \cite{AR}. 

\section{Appendix: 2-CY triangulated categories of finite type}\label{finite}
In this section, we consider a more general situation than in section
\ref{additive}.
Let $k$ be an algebraically closed field and
$\mathcal{C}$ a $k$-linear connected 2-Calabi-Yau triangulated category
with only finitely many indecomposable objects.
We show that it follows from the shape of
the AR quiver of $\mathcal{C}$ whether cluster tilting objects
(respectively, non-zero rigid objects) exist in $\mathcal{C}$ or not.
Let us start with giving the possible shapes of the AR quiver of $\mathcal{C}$.
Recall that a subgroup $G$ of ${\rm Aut}({\bf Z}\Delta)$ is called {\it weakly admissible} if $x$ and $gx$ do not have a common direct successor for any vertex $x$ in ${\bf Z}\Delta$ and $g\in G\backslash\{1\}$ \cite{XZ,Am}.

\begin{proposition}\label{list}
The AR quiver of $\mathcal{C}$ is ${\bf Z}\Delta/G$
for a Dynkin diagram $\Delta$ and a weakly admissible subgroup $G$ of
${\rm Aut}({\bf Z}\Delta)$ which contains $F\in{\rm Aut}({\bf Z}\Delta)$
defined by the list below. Moreover, $G$ is generated by a single element
$g\in{\rm Aut}({\bf Z}\Delta)$ in the list below.
\[\begin{array}{|c|c|c|c|}\hline
\Delta&{\rm Aut}({\bf Z}\Delta)&F&g\\ \hline\hline
(A_n)\ n:\mbox{odd}&{\bf Z}\times{\bf Z}/2{\bf Z}&(\frac{n+3}{2},1)&
(k,1)\ (k|\frac{n+3}{2},\ \frac{n+3}{2k}\mbox{ is odd})\\ \hline
(A_n)\ n:\mbox{even}&{\bf Z}&n+3&
k\ (k|n+3)\\ \hline
(D_n)\ n:\mbox{odd}&{\bf Z}\times{\bf Z}/2{\bf Z}&(n,1)&
(k,1)\ (k|n)\\ \hline
(D_4)&{\bf Z}\times S_3&(4,0)&
(k,\sigma)\ (k|4,\ \sigma^{\frac{4}{k}}=1)\\ \hline
(D_n)\ n:\mbox{even},\ n>4&{\bf Z}\times{\bf Z}/2{\bf Z}&(n,0)&
(k,0)\ (k|n)\mbox{ or }(k,1)\ (k|n,\ \frac{n}{k}\mbox{ is even})\\ \hline
(E_6)&{\bf Z}\times{\bf Z}/2{\bf Z}&(7,1)&
(1,1)\mbox{ or }(7,1)\\ \hline
(E_7)&{\bf Z}&10&
1,2,5\mbox{ or }10\\ \hline
(E_8)&{\bf Z}&16&
1,2,4,8\mbox{ or }16\\ \hline
\end{array}\]
In each case, elements in the torsion part of ${\rm Aut}({\bf Z}\Delta)$
are induced by the automorphism of $\Delta$.
The torsionfree part of ${\rm Aut}({\bf Z}\Delta)$ is generated by $\tau$
except the case $(A_n)$ with even $n$, in which case it is generated by
the square root of $\tau$.
\end{proposition}

\begin{proof}
By \cite{XZ} (see also \cite[4.0.4]{Am}),
the AR quiver of $\mathcal{C}$ is ${\bf Z}\Delta/G$
for a Dynkin diagram $\Delta$ and a weakly admissible subgroup $G$ of
${\rm Aut}({\bf Z}\Delta)$.
Since $\mathcal{C}$ is 2-Calabi-Yau, $G$ contains $F$.
By \cite[2.2.1]{Am}, $G$ is generated by a single element $g$.
By the condition $F\in\langle g\rangle$, we have the above list.
\end{proof}

Note that, by a result of Keller \cite{Ke}, 
the translation quiver ${\bf Z}\Delta/G$
for any Dynkin diagram $\Delta$ and any weakly admissible group $G$ of ${\rm Aut}({\bf Z}\Delta)$
is realized as the AR quiver of a triangulated orbit category $\mathcal{D}^b(H)/g$
for a hereditary algebra $H$ of type $\Delta$ and some autofunctor $g$ of $\mathcal{D}^b(H)$.

\begin{theorem}
(1) $\mathcal{C}$ has a cluster tilting object if and only if
the AR quiver of $\mathcal{C}$ is ${\bf Z}\Delta/g$
for a Dynkin diagram $\Delta$ and
$g\in{\rm Aut}({\bf Z}\Delta)$ in the list below.
\[\begin{array}{|c|c|c|}\hline
\Delta&{\rm Aut}({\bf Z}\Delta)&g\\ \hline\hline
(A_n)\ n:\mbox{odd}&{\bf Z}\times{\bf Z}/2{\bf Z}&
(\frac{n+3}{6},1)\ (3|n)\mbox{ or }(\frac{n+3}{2},1)\\ \hline
(A_n)\ n:\mbox{even}&{\bf Z}&
\frac{n+3}{3}\ (3|n)\mbox{ or }n+3\\ \hline
(D_n)\ n:\mbox{odd}&{\bf Z}\times{\bf Z}/2{\bf Z}&
(k,1)\ (k|n)\\ \hline
(D_4)&{\bf Z}\times S_3&
(k,\sigma)\ (k|4,\ \sigma^{\frac{4}{k}}=1,\ (k,\sigma)\neq(1,1))\\ \hline
(D_n)\ n:\mbox{even},\ n>4&{\bf Z}\times{\bf Z}/2{\bf Z}&
(k,\overline{k})\ (k|n)\\ \hline
(E_6)&{\bf Z}\times{\bf Z}/2{\bf Z}&
(7,1)\\ \hline
(E_7)&{\bf Z}&
10\\ \hline
(E_8)&{\bf Z}&
8\mbox{ or }16\\ \hline
\end{array}\]

(2) $\mathcal{C}$ does not have a non-zero rigid object if and only if
the AR quiver of $\mathcal{C}$ is ${\bf Z}\Delta/g$
for a Dynkin diagram $\Delta$ and
$g\in{\rm Aut}({\bf Z}\Delta)$ in the list below.
\[\begin{array}{|c|c|c|}\hline
\Delta&{\rm Aut}({\bf Z}\Delta)&g\\ \hline\hline
(A_n)\ n:\mbox{odd}&{\bf Z}\times{\bf Z}/2{\bf Z}&
-\\ \hline
(A_n)\ n:\mbox{even}&{\bf Z}&
1\\ \hline
(D_n)\ n:\mbox{odd}&{\bf Z}\times{\bf Z}/2{\bf Z}&
-\\ \hline
(D_4)&{\bf Z}\times S_3&
(1,1)\\ \hline
(D_n)\ n:\mbox{even},\ n>4&{\bf Z}\times{\bf Z}/2{\bf Z}&
(1,0)\\ \hline
(E_6)&{\bf Z}\times{\bf Z}/2{\bf Z}&
(1,1)\\ \hline
(E_7)&{\bf Z}&
1\\ \hline
(E_8)&{\bf Z}&
1\mbox{ or }2\\ \hline
\end{array}\]
\end{theorem}

\begin{proof}
Our method is based on the computation of additive functions in
section \ref{additive}. We refer to \cite[Section 4.4]{I1} for detailed explanation.

(1) Assume that $g$ is on the list.
Then one can check that $\mathcal{C}$ has a cluster tilting object.
For example, consider the $(D_n)$ case here. Fix a vertex $x\in{\bf Z}\Delta$
corresponding to an end point of $\Delta$ which is adjacent
to the branch vertex of $\Delta$. Then the subset
$\{(1,1)^lx\ |\ l\in{\bf Z}\}$ of ${\bf Z}\Delta$ is stable under the action
of $g$, and gives a cluster tilting object of $\mathcal{C}$.

Conversely, assume that $\mathcal{C}$ has a cluster tilting object.
Then one can check that $g$ is on the list.
For example, consider the $(A_n)$ case with even $n$ here.
By \cite{CCS,I1}, cluster tilting objects correspond to
dissections of a regular $(n+3)$-polygon into triangles
by non-crossing diagonals. The action of $g$ shows that
it is invariant under the rotation of $\frac{2k\pi}{n+3}$-radian.
Since the center of the regular $(n+3)$-polygon is contained in
some triangle or its edge, we have $\frac{2k\pi}{n+3}=2\pi,\ \frac{4\pi}{3},
\ \pi$ or $\frac{2\pi}{3}$.
Since $k|n+3$ and $n$ is even, we have $k=n+3$ or $\frac{n+3}{3}$.

(2) If $g$ is on the list above, then one can easily check that $\mathcal{C}$
does not have non-zero rigid objects.
Conversely, if $g$ is not on the list, then one can easily check that
at least one indecomposable object which corresponds
to an end point of $\Delta$ is rigid.
\end{proof}

\end{document}